\documentclass[12pt,a4paper,oneside]{amsart}
\usepackage{csquotes, amsmath,amscd, latexsym, mathptmx, courier, verbatim}
\usepackage{parskip}
\usepackage{pdfpages,amsfonts,amsthm,amssymb,latexsym,ucs,graphicx,leftidx,fancyhdr}
\usepackage[left=2.5cm,right=2.5cm,bottom=2cm,top=2cm]{geometry} 

\usepackage{color,amssymb,latexsym,amsfonts,textcomp,fancyhdr,calc,graphicx}
\usepackage{pdfpages,amsfonts,amsthm,amssymb,latexsym,graphicx,leftidx,fancyhdr}
\usepackage{amsmath,amstext,amsthm,amssymb,amsxtra}
\usepackage[colorlinks,citecolor=cyan, pagebackref,hypertexnames=false]{hyperref}

\usepackage{dsfont}
\usepackage[framemethod=tikz]{mdframed}
\usepackage{asymptote}

\usepackage{cite}

\usepackage{txfonts,pxfonts,tikz} 
\usepackage{tgschola}
\usepackage[T1]{fontenc}

\usepackage{mathtools}
\mathtoolsset{showonlyrefs}

\newcommand{\beql}[1]{\begin{equation}\label{#1}}
\newcommand{\eeq}{\end{equation}}

\theoremstyle{plain}

\newtheorem{theorem}{Theorem}[section]
\newtheorem{proposition}[theorem]{Proposition}
\newtheorem{lemma}[theorem]{Lemma}
\newtheorem{corollary}[theorem]{Corollary}

\theoremstyle{definition}

\newtheorem{definition}[theorem]{Definition}
\newtheorem{remark}[theorem]{Remark}

\newtheorem{problem}[theorem]{Problem}

\numberwithin{equation}{section}
\usepackage[shortlabels]{enumitem}

\newcounter{mysec}
\newcounter{mysubsec}[mysec]
\definecolor{citecolour}{rgb}{0.0, 0.0, 0.8}
\definecolor{urlcolour}{rgb}{1,0.5,0}
\colorlet{linkcolour}{green!50!black}
\hypersetup{colorlinks,breaklinks,
		linkcolor=linkcolour,citecolor=citecolour,
		filecolor=linkcolour, urlcolor=urlcolour}
\usepackage{version}

\def\G{\mathfrak{G}}

\newenvironment{proofof}{{\bf {Proof.} }}{\hfill $\blacksquare$ \\}
\newenvironment{proofofclaim}{{\bf {Proof of Claim.} }}{\hfill $\spadesuit$ \\}

\renewcommand{\leq}{\leqslant}
\renewcommand{\geq}{\geqslant}

\begin{document}
\title[On the common transversal probability]{On the common transversal probability in finite groups}

\author{S. Aivazidis}
\address{Department of Mathematics \& Applied Mathematics, University of Crete, Greece}
\email{s.aivazidis@uoc.gr}

\author{M. Loukaki}
\address{Department of Mathematics \& Applied Mathematics, University of Crete, Greece}
\email{mloukaki@uoc.gr}

\author{T.\,W. M\"uller}
\address{ Department of Mathematics, University of Vienna, Austria
(formerly of Queen Mary \& Westfield College, London, UK)}
\email{muellet4@univie.ac.at}

\thanks{The first author is partially supported by the Hellenic Foundation for Research and Innovation, Project HFRI-FM17-1733.}

\begin{abstract}
 Let $G$ be a finite group, and let $H$ be a subgroup of $G$. We compute the probability, denoted by $P_G(H)$, that a left transversal of $H$ in $G$ is also a right transversal, thus a two-sided one. Moreover, we define, and denote by $\mathrm{tp}(G)$, the common transversal probability of $G$ to be the minimum, taken over all subgroups $H$ of $G$, of $P_G(H)$. We prove a number of results regarding the invariant $\mathrm{tp}(G)$, like lower and upper bounds, and possible values it can attain. We also show that $\mathrm{tp}(G)$ determines structural properties of $G$. Finally, several open problems are formulated and discussed.
\end{abstract}

\keywords{Finite groups, Two-sided Transversals, Transversal Probability}

\subjclass[2020]{05D15, 05E16, 20D10, 20D60, 20P05}

\maketitle

\tableofcontents

\setlength{\parskip}{0.5em}

\section{Introduction}\label{Sec:Intro}
\noindent We begin with some general remarks and a little bit of history. 
It is often the case that probability theory interacts with group theory. 
The area that we now call \enquote{probabilistic group theory} may reasonably
be said to begin with a series of papers by Erd\H{o}s and Tur\'{a}n---stretching 
from $1965$ to $1972$ and beginning with~\cite{erdos}---where statistical properties 
of the symmetric group are examined in detail.
For survey articles of this area we refer the interested reader
to~\cite{shalev98}, \cite{shalev01}, \cite{pgth}, \cite{pgthcc}, as well as the recent~\cite{mann18}.
Prominent concepts that probabilistic group theory concerns itself with
include the so-called \enquote{probabilistic generation} and the \enquote{commuting probability}
as well as variants of the latter.
The former is about questions of the type:
\begin{quote}
Given a family of groups $\left\{G_n\right\}_n$, what is the probability that a $d$--tuple of elements from $G_n$, chosen uniformly at random, generates $G_n$ as $n \to \infty$?
\end{quote}
Recent work in this direction has mainly focussed on the so-called
$(2,3)$-generation problem for finite simple groups.
The question here is whether all finite simple groups
can be generated by an involution and an element
of order $3$.
\noindent The commuting probability of $G$, often denoted by $\mathrm{cp}(G)$, 
is the probability that two elements of $G$, chosen uniformly at random, commute.
It was popularised by Gustafson~\cite{gustafson} who, in turn, 
traces its origin to Erd\H{o}s and Tur{\'a}n's 
series of papers on the statistics of the symmetric group.
Using the class equation one can prove that 
$\mathrm{cp}(G) = k(G)/|G|$, 
where $k(G)$ is the number of conjugacy classes of $G$. 
Gustafson also established the gap result stating 
that, if $G$ is non-abelian then $\mathrm{cp}(G) \leq 5/8$, 
with equality if, and only if, $G/Z(G) \cong C_2 \times C_2$. 
Of course, $G$ is abelian if, and only if, $\mathrm{cp}(G) = 1$ and, 
despite the $\left(5/8, 1\right)$ gap, the commuting probability may be viewed 
as an arithmetic quantification of the \enquote{abelianness} of a group.
\noindent Certain other variants of the commuting probability have also
received attention. We take this opportunity to mention only a
couple of those: T{\u{a}}rn{\u{a}}uceanu's~\cite{tuarnuauceanu09} 
concept of the \enquote{subgroup permutability degree} as the 
probability that two subgroups of $G$ permute.
It turns out that this probability is an arithmetic measure of how close $G$ is to being an Iwasawa group
or, equivalently, a nilpotent modular group.
A second interesting variant is due to Blackburn, Britnell, and Wildon.
In~\cite{BBW} the authors introduce the probability that two elements, 
chosen independently and at random from a (finite) group $G$, are conjugate
and study many of the fundamental properties that this probability enjoys.

\noindent We now focus our attention on the contents of the paper at hand.
Let $G$ be a finite group, and let $H$ be a subgroup of $G$.    
By Hall's Marriage Theorem (see~\cite{PHall}) there exists a common set 
of representatives for the left and the right cosets 
of $H$ in $G$; such a set is called a double or two-sided transversal of $H$ in $G$. 
Even though the existence of such a double transversal is guaranteed, 
it seems unlikely that (in general) a left transversal of a non-normal subgroup 
will also be a right transversal. 
One of the  main objectives of this paper is to compute the probability 
that a randomly chosen (with the uniform distribution) left transversal of $H$ in $G$ is also a right transversal. 
If the index of $H$ in $G$ is $n$ and we denote by $\mathrm{DT}_G(H)$ the set 
of all double transversals of $H$ in $G$, 
then the quotient 
\[
P_G(H) \coloneqq \frac{\vert \mathrm{DT}_G(H)\vert}{\vert H\vert^n}
\]
is precisely the (Laplacian) probability in question.
Actually, what we compute is the above probability in a more general setting, 
where two subgroups $H$, $K$ of the same index are given, 
and we are seeking left transversals of $H$ that are also right transversals for $K$;  
see Proposition~\ref{Prop:PFormula}.

\noindent In the present article, and in the spirit of $\mathrm{cp}(G)$, 
a novel invariant is introduced and studied: 
the transversal probability $\mathrm{tp}(G)$ of $G$ defined as 
\[
\mathrm{tp}(G) \coloneqq \min_{H\leq G}\, P_G(H).
\]

\noindent Our objective here is two-fold: firstly, to compute $P_G(H)$ and to discuss its various numerical properties. 
In this direction, for example, we prove that if $n = (G:H)$, where $H$ is not normal in $G$,  
then (see Corollary~\ref{Cor:Prod}) 
\[
\frac{(n-1)!}{(n-1)^{n-1}}\,\leq\, P_G(H) \,\leq\, \frac{1}{2}, 
\]
and (see Theorem~\ref{Thm:Main})
\[
\lim_{n \to \infty} P_G(H) = 0;
\]
secondly, to investigate how $\mathrm{tp}(G)$ influences the structure of the group $G$. Results in this direction are for example the following:
\begin{itemize}
\item If  $\mathrm{tp}(G) > (1/2)^{40}$ then either $G$ is soluble 
or has a section isomorphic to $A_5$. 
(Proposition~\ref{Prop:solubilitycriterion})
\vspace{1.5mm}
\item If $\mathrm{tp}(G) > (1/2)^8$ then either $G$ is supersoluble 
or has a section isomorphic to $A_4$. 
(Theorem~\ref{Thm:supersol})
\vspace{1.5mm}
\item If  $\mathrm{tp}(G) > (2/9)^2$  then either $G$ is nilpotent or has a section isomorphic to one of the groups $\{ A_4, D_3, D_5, D_7  \}$.   (Theorem~\ref{Thm:nilp})
\vspace{1.5mm}
\item If  $\mathrm{tp}(G) > 4/81$ and $G$ is non-abelian then it has derived length 2. (Proposition~\ref{Prop:derived})
\end{itemize}

\noindent Furthermore, the results above are all sharp. 
In addition, the upper bound of $1/2$ obtained for $P_G(H)$  
which holds for every non-normal $H \leq G$ 
carries over to $\mathrm{tp}(G)$ for every non-Dedekind group $G$;
that is, $\mathrm{tp}(G) \leq 1/2$ for every such group. 
We are able to fully characterise groups $G$ with $\mathrm{tp}(G) = 1/2$, 
as well as those with $\mathrm{tp}(G) = 1/4$ 
(cf. Corollary~\ref{Cor:tp = 1/2} and Theorem~\ref{thm:tp = 1/4}).

\noindent Finally, while investigating the possible values of $\mathrm{tp}(-)$, 
the following number theoretic result, of independent interest, was obtained (see Theorem~\ref{Thm:prodp_i}):

\begin{itemize}
\item If $\prod_{i = 1}^n\frac{t_i!}{t_i^{t_i}} \,=\, \prod_{i = 1}^k\frac{p_i!}{p_i^{p_i}}$
where the numbers $t_i >1$ are positive integers and the $p_i$ are distinct primes, 
then $k = n$ and, after appropriate rearrangement, $t_i = p_i$ for all $i = 1, \ldots, n$. 
\end{itemize}

\noindent The article is organised as follows. In Section~\ref{Sec:CosetGraph} the coset intersection graph is introduced 
and its properties needed for the  computation of $P_G(-, -)$  are analyzed. 
In Section~\ref{Sec:CTProb} the function $P_G(-,-)$ is introduced and its precise value is calculated. 
A connection with the permanent of a matrix is established and several bounds on its values are given. 
In addition, the behaviour of $P_G(-)$ with respect to subgroups and quotients is determined. 
In Section~\ref{Sec:tp} we work with the function $\mathrm{tp}(G)$, discussing the values it can and cannot attain, 
and the way it interacts with the structure of the group $G$. 
Finally, in Section~\ref{Sec:Problems} some open problems and questions are posed and discussed.

\section{Coset intersection graphs}
\label{Sec:CosetGraph}

\noindent In this section we introduce the concept of a coset intersection graph, 
which is the required background for the computation of $P_G(-,-)$.  
We follow the definition in~\cite{BCL}. 

\begin{definition}
Let $G$ be a group, and let $H, K \leq G$ be subgroups. 
The \emph{coset intersection graph} $\Gamma = \Gamma^G_{H, K}$ 
has vertex set $V(\Gamma) = G/H \sqcup K\backslash G,$ 
two vertices being connected by an (undirected) edge if, and only if, 
the corresponding cosets have non-empty intersection. 
If $H = K,$ we set $\Gamma^G_H \coloneqq \Gamma^G_{H, H}$. 
\end{definition}
\noindent By definition, $\Gamma^G_{H, K}$ is a bipartite graph, split between $\{l_i H\}_{i\in I}$ and $\{Kr_j\}_{j\in J}$, where $\{l_i\}_{i\in I}$ is a left transversal for $H$ in $G$ and $\{r_j\}_{j\in J}$ is a right transversal for $K$ in $G$. In this paper, $G$ will always be a finite group; thus, the coset intersection graphs considered in what follows will all be finite. We also denote  by $K_{a,b}$ the complete bipartite graph on  $a$ and $b$ vertices, for any positive integers $a$ and $b$. 

\begin{lemma}
\label{Lem:CompComplete}
Let $\Delta \leq \Gamma^G_{H, K}$ be a connected component. Then
\vspace{-2mm}
\begin{enumerate}
\item[(i)] $\bigcup_{gH \in V(\Delta)} gH = \bigcup_{Kg'\in V(\Delta)} Kg';$
\vspace{2mm}
\item[(ii)] $\Delta$ is a complete bipartite graph. 
\end{enumerate}
\end{lemma}

\begin{proofof}
(i) For given $g\in G$ with $gH\in V(\Delta)$, we have 
\[
g H \,\subseteq\, \underset{gH \cap Kg'\neq \emptyset}{\bigcup_{Kg'\in K\backslash G}} Kg' \,\subseteq\, \bigcup_{Kg'\in V(\Delta)} Kg',
\]
so that the left-hand side of (i) is contained in the right-hand side, and a similar argument establishes the reverse inclusion.

\noindent (ii) This is \cite[Thm.~3]{BCL}.
\end{proofof}

\noindent Our next result, which (in somewhat different language and with a less direct proof) goes back to Ore \cite{Ore}, computes the number of connected components of a coset intersection graph. 

\begin{lemma}\label{Lem:sDet}
The number of connected components of the graph $\Gamma^G_{H, K}$ equals the number 
$\vert K\backslash G/H\vert$ of $(K, H)$-double cosets of $G$.
\end{lemma}

\begin{proofof}
Denote by $\mathcal{C}(\Gamma)$ the collection of all connected components of $\Gamma = \Gamma^G_{H, K}$. If $aH, Kb\in V(\Delta)$ for some connected component $\Delta$ of $\Gamma$ then, by Part~(ii) of Lemma~\ref{Lem:CompComplete}, there exist elements $h\in H$ and $k\in K$, such that $ah = kb$, or $b = k^{-1}ah$. Hence, $KbH = KaH$. Keeping $aH\in V(\Delta)$ fixed while running through all right vertices $Kb$ of $\Delta$, we thus see that
\[
KbH = KaH,\quad (Kb \in V(\Delta)).
\]
A symmetric argument, this time keeping a right vertex $Kb\in V(\Delta)$ fixed, and running through the left vertices $aH$ of $\Delta$ shows that
\[
KaH = KbH,\quad (aH \in V(\Delta)).
\]
It follows that the whole component $\Delta$ (meaning every vertex of $\Delta$) is contained in one and the same $(K, H)$-double coset $Kg_\Delta H$, and sending $\Delta$ to $K g_\Delta H$ yields a well-defined map 
\[
c: \mathcal{C}(\Gamma^G_{H, K})\,\longrightarrow\, K\backslash G / H,\quad c(\Delta) = KgH,\,\,(gH \in V(\Delta)).
\]
Conversely, let $KgH \in K\backslash G/H$ be a given $(K, H)$-double coset. Then each left $H$-coset $kgH$ and every right $K$-coset $Kgh$ contained in $KgH$ intersect non-trivially, as $kgh \in kgH \cap Kgh$. An argument analogous to the one given above now shows that all left $H$-cosets and all right $K$-cosets contained in $KgH$ lie in one and the same connected component $\Delta_g$ of $\Gamma$, and sending $KgH$ to $\Delta_g$ gives a well-defined map 
\[
d: K\backslash G/H\,\longrightarrow\, \mathcal{C}(\Gamma^G_{H, K}),\quad d(KgH) = \Delta,\,\,(gH\in V(\Delta)).
\]
The fact that $d(c(\Delta)) = \Delta$ for $\Delta\in\mathcal{C}(\Gamma)$ is now obvious (pin down $\Delta$ by means of a vertex $gH\in V(\Delta)$) and, similarly, we find that $c(d(KgH)) = KgH$, finishing the proof.
\end{proofof}

\noindent With every edge of a coset intersection graph we associate a weight in the following natural way.

\begin{definition}
\label{Def:EdgeWeight}
Given a coset intersection graph $\Gamma = \Gamma^G_{H, K},$ we associate to an edge $e = l_i H - K r_j$ in $\Gamma$ the \emph{weight} 
$w(e) \coloneqq \, \vert l_iH \cap Kr_j\vert$. 
\end{definition}

\noindent If $g\in G$ and if $\Delta\leq \Gamma$ is a connected component of the coset intersection graph $\Gamma = \Gamma^G_{H, K}$, we shall write $g\in\Delta$ to mean $gH \in  V(\Delta)$. Observe that,  as $g \in gH \cap Kg$,
we clearly have $gH \in V(\Delta)$ if and only if $Kg \in V(\Delta)$. 

\begin{lemma}
\label{Lem:EqualWeights}
Let $\Delta \leq \Gamma^G_{H, K}$ be a connected component. Then all edges of $\Delta$ carry the same weight $w,$ and we have
\[
w = \vert gHg^{-1} \cap K\vert = \vert H \cap g^{-1} K g\vert,\quad (g\in \Delta).
\]
\end{lemma}
\begin{proofof}
Let $a H, K b\in V(\Delta)$ and  $g\in  aH \cap Kb$. Then $aH = gH$ and $Kb = Kg$, and thus
\[
 w_{a, b} \coloneqq w(aH - Kb) 
 = \vert aH \cap Kb\vert 
 = \vert gH \cap Kg\vert
 = \vert gH g^{-1} \cap K\vert. 
 \]
Keeping $aH$ fixed, consider an arbitrary right coset $Kc\in V(\Delta)$, and set $w_{a, c} \coloneqq w(aH - Kc)$. Then, by the previous computation, $w_{a, c} = \vert xHx^{-1} \cap \nolinebreak K\vert$ for any $x\in aH \cap Kc$. Since $x \in aH = gH$, we have $x = gh'$ for some $h'\in H$, so that 
\[
w_{a, c} = \vert g h' H (h')^{-1} g^{-1} \cap K\vert = \vert gHg^{-1} \cap K\vert = w_{a, b}.
\]
Next, consider an arbitrary left coset $dH\in V(\Delta)$, noting that, at this stage, $dH - Kc$ is an arbitrary edge in $\Delta$. Changing the roles of $H$ and $K$ in the previous  argument we also get $w_{a, c} = w_{d, c} $. Hence  
\[
w_{d, c} =w_{a, c} = w_{a, b} = w,
\]
and our result follows.
\end{proofof}

\begin{definition}
\label{Def:CompWeight}
If $\Delta \leq \Gamma^G_{H, K}$ is a connected component, then we call the common weight of the edges in $\Delta$ the \emph{weight} of $\Delta,$ denoted by $w(\Delta)$. 
\end{definition}

\begin{lemma}
\label{Lem:HKWeight}
\begin{enumerate}
\item[(a)] If $\Delta_\sigma \cong K_{s_\sigma,\, t_\sigma}$ is a connected component of $\Gamma^G_{H, K},$ then $\vert H\vert = w_\sigma t_\sigma$ and $\vert K\vert = w_\sigma s_\sigma,$ where $w_\sigma = w(\Delta_\sigma)$ is the weight of $\Delta_\sigma$. In particular, $s_\sigma/t_\sigma = \vert K\vert/\vert H\vert,$ and we have $s_\sigma = t_\sigma$ provided that $\vert H\vert = \vert K\vert$. In the latter case, the number $m$ of components of type $K_{1, 1}$ is given by
\begin{equation}
\label{Eq:mForm}
m = \begin{cases} (N_G(H) : H),&\mbox{if $H$ is conjugate to $K$ in $G,$}\\
0, & \mbox{otherwise,}\end{cases}
\end{equation}
and the weight of such a component equals $|H|$.
\vspace{2mm}
\item[(b)] If $G = \bigsqcup_{\sigma = 1}^s K g_\sigma H,$ then $t_\sigma = (H:H\cap K^{g_\sigma}),$ where $\Delta_\sigma \cong K_{t_\sigma, t_\sigma}$ is the connected component of $\Gamma^G_{H, K}$ containing $g_\sigma$. 
\end{enumerate} 
\end{lemma}

\begin{proof}
(a) Let $\Delta_\sigma \cong K_{s_\sigma,\,t_\sigma}$ be a connected component of the coset intersection graph $\Gamma^G_{H,K}$ of weight $w(\Delta_\sigma) = w_\sigma$. By Lemma~\ref{Lem:EqualWeights}, each given left coset $L_i$ of $\Delta_\sigma$ intersects every right coset $R_j\in V(\Delta_\sigma)$ in exactly $w_\sigma$ elements and, as the right cosets of $\Delta_\sigma$ are pairwise disjoint, we have 
\[
\left\lvert L_i \,\cap \bigcup_{R_j\in V(\Delta_\sigma)} R_j\right\rvert = w_\sigma t_\sigma,\quad (L_i\in V(\Delta_\sigma)).
\]
However, $L_i \subseteq \bigcup_{R_j\in V(\Delta)} R_j$ by Lemma~\ref{Lem:CompComplete}(i), so that in fact 
\[
L_i \,\cap \bigcup_{R_j\in V(\Delta)} R_j = L_i, 
\]
and we find that 
\[
\vert H\vert = \vert L_i\vert = w_\sigma t_\sigma,
\] 
as claimed. A symmetric argument yields that $\vert K\vert = w_\sigma s_\sigma$. As concerns the assertion about the number $m$ of trivial components, we note that $1= t_{\sigma}= s_{\sigma} $ precisely when $w_\sigma = \vert H \vert = \vert K \vert $, which in turn implies (by Lemma~\ref{Lem:EqualWeights}) that $K=  gHg^{-1} \cap K $ and $ H = H \cap g^{-1} K g$ for $g \in \Delta_\sigma$. Hence  $H, K$  are congugate  and  
\[
m = \big\vert\big\{g\in G:\, H^g = K\big\}\big\vert \big/ \big\vert H\big\vert,
\]
from which the given formula follows. The remaining assertions are now clear.

\noindent (b) This is a special case of \cite[Prop.~6]{BCL}.
\end{proof}

\noindent We record upper and lower bounds for the number $s = \vert K\backslash G/H\vert$ of connected components of a symmetric coset intersection graph $\Gamma^G_{H, K}$.

\begin{lemma}
\label{Lem:sBounds}
Let $H, K\leq G$ be such that $(G:H) = n = (G:K),$ and let $s = \vert K\backslash G/H\vert$. Then we have 
\begin{equation}
\label{Eq:sEst}
\frac{n-m}{\vert H\vert}\,+\, m\, \leq\,s \leq \frac{n-m}{p} \,+\, m,
\end{equation}
where $p$ is the smallest prime divisor of $\vert H\vert,$ and $m$ is given by {\em \eqref{Eq:mForm}}. 
\end{lemma} 

\begin{proofof}
Let 
\[
G = \bigsqcup_{\sigma = 1}^s \,K g_\sigma H, 
\]
where $g_\sigma \in\Delta_{\sigma}$ for $1\leq \sigma \leq s$. We observe that 
\[
\vert K g_\sigma H\vert = \frac{\vert H\vert\cdot \vert K\vert}{\vert H \,\cap\, K^{g_\sigma}\vert} = \frac{\vert H\vert^2}{\vert H \,\cap\, K^{g_\sigma}\vert},
\]
and that 
\[
\vert H\vert \,\leq \, \frac{\vert H\vert^2}{\vert H \,\cap\, K^{g_\sigma}\vert} \leq \vert H\vert ^2. 
\]
Here, the left-hand side is assumed if, and only if, $g_\sigma\in\{g\in G: K^g = H\}$, while the right-hand side is assumed if, and only if, $H \cap K^{g_\sigma} = 1$. It follows that 
\[
n|H| = |G| = \sum_{\sigma = 1}^{s} |K g_\sigma H| = m|H| \,+\, 
\sum_{\sigma = m+1}^{s} |K g_\sigma H| \leq m|H| \,+\, (s-m)|H|^2.
\]
Dividing both sides by $|H|$ yields $n \leq m +(s-m)|H|$
or equivalently 
\[
\frac{n-m}{|H|} \,+\, m \leq s.
\]
This proves the first half of the assertion. 
The upper bound for $s$ is obtained in a similar way:
\begin{align*}
n|H| = |G| & = \, \sum_{\sigma = 1}^{s} |K g_\sigma H| \\
& = m|H| \,+\, 
\sum_{\sigma = m+1}^{s} |H|\cdot (H : H \cap K^{g_\sigma}) \\
&\geq\, m|H| \,+\, (s-m)p|H|.
\end{align*}
Dividing both sides by $|H|$ yields $n \geq m +(s-m)p$,
or equivalently 
\[
\frac{n-m}{p} \,+\, m \geq s,
\]
completing the proof.
\end{proofof}

\noindent With the help of Lemma~\ref{Lem:sBounds}, we can characterise groups $G$ having a non-normal subgroup $H$ satisfying $\vert H\backslash G/H\vert = 2$, as the next proposition shows. 

\begin{proposition}
Let $G$ be a finite group and let $H \leq G$ be non-normal
and of index $n$ in $G$. 
If $s = \vert H\backslash G/H\vert = 2$ and $|H| = n - 1$, then G is a Frobenius group
and $H$ is a Frobenius complement.
Conversely, suppose that G is a Frobenius group
and that $H$ is a Frobenius complement.
Then $s = 2$ if, and only if, $|H| = n-1$.
\end{proposition}

\begin{proofof}
By Lemma~\ref{Lem:sBounds} with $H = K$, we have
\[
2 \geq \frac{n-m}{|H|} \,+\, m, 
\]
which implies that $m = 1$, since $H$ is not normal in $G$ (i.e., $m < n$); thus, $N_G(H) = H$.
Now let $g \in G - H$. 
Since $s = 2$, we have $G = H \sqcup HgH$, and therefore 
\[
|G| = |H| \,+\, \frac{|H|^2}{|H \cap H^g|}.
\]
Consequently, for $s = 2$ and $\vert H\vert = n-1$,
\[
n = 1 \,+\, \frac{|H|}{|H \cap H^g|} = 1 \,+\, \frac{n-1}{|H \cap H^g|}.
\]
It follows that $|H \cap H^g| = 1$, thus $H \cap H^g = 1$.
Since $g$ was arbitrary subject to lying outside $H$,
we deduce that $G$ is a Frobenius group with Frobenius complement $H$.

\noindent Suppose now that $G$ is a Frobenius group, and that $H$ is a Frobenius complement. Then we have $s = 2$ if, and only if, $G = H \sqcup HgH$ for every $g \in G - H$
or, equivalently, if, and only if,
\[
|G| = |H| + |H|^2; 
\]
that is, if, and only if, $|H| = n - 1$, which is as desired. 
\end{proofof}

\section{The common transversal probability $P_G(-,-)$ of given subgroups.} \label{Sec:CTProb}


\noindent
Let $G$ be a finite group, and let $H\leq G$ be a subgroup. As we have noted before,  a given left transversal for $H$ in $G$ may or may not be a two-sided one. The aim of this section is to compute the probability $P_G(H)$ that this happens. We actually give a generalized definition of the  above probability, where  two groups $H,K$ of the same index $n$ are concerned,  and  compute its value. This way we are able to show that the bigger the index $n$ is, the smaller the probability of a left transversal of $H$ to be a right transversal of $K$ is (see Theorem~\ref{Thm:Main} below). In addition, we give bounds for the values of $P_G(H, K)$ and  associate $P_G(H, K)$ with a permanent of a doubly stochastic matrix.  Finally in the last  subsection we show that $P_G(-)$ behaves well with respect to subgroups and homomorphic images.

\subsection{Definition and computation of  $P_G(-,-)$.}
We start with a definition of the main actors of this paper.
\begin{definition}
\label{Def:CommProb}
Let $G$ be a finite group, and let $H, K\leq G$ be subgroups such that $(G:H) = n = (G:K)$.
\vspace{-2mm}
\begin{enumerate}
\item[(i)] We denote by $\mathrm{DT}_G(H, K)$ the set of all left transversals for $H$ in $G$, which are also right transversals for $K$ in $G$. 
\vspace{1.5mm}
\item[(ii)] We let 
\[
P_G(H,K) \coloneqq \frac{\vert \mathrm{DT}_G(H, K)\vert}{\vert H\vert^n},
\]
so that $P_G(H, K)$ is the (Laplacian) probability that a left transversal for $H$ in $G$ is also a right transversal for $K$ in $G$. If $H = K$, we let $P_G(H) \coloneqq P_G(H, H)$ be the probability that a left transversal for $H$ in $G$ is a two-sided one. We call $P_G(H, K)$ the \emph{common transversal probability} for the subgroups $H, K\leq G$. 
\vspace{1.5mm}
\item[(iii)] We set
\[
\mathrm{tp}(G) \coloneqq \min_{H\leq G}\, P_G(H),
\]
which is an invariant of $G$ alone, termed the \emph{common transversal probability} of the group $G$.
\end{enumerate}
\end{definition}

\noindent The first result of this section computes the probability $P_G(H, K)$ in terms of the sizes of the connected components of the coset intersection graph $\Gamma^G_{H, K}$. 

\begin{proposition}
\label{Prop:PFormula}
Let $G$ be a finite group, and let $H, K\leq G$ be subgroups such that 
\[
(G:H) = (G:K) = n.
\] 
Assume further that $\vert K\backslash G/H\vert = s$ and let  $\Delta_\sigma = K_{t_\sigma, t_\sigma}$ for $\sigma \in [s]$ denote  the connected components of   $\Gamma = \Gamma^G_{H, K}$. 
Then
\begin{equation}
\label{Eq:PFormula}
P_G(H, K) = \prod_{\sigma = 1}^s \frac{t_\sigma!}{t_\sigma^{t_\sigma}},
\end{equation}
where $t_\sigma \,\big\vert\, \vert H\vert$ and $\sum_{\sigma = 1}^s t_\sigma = n$.
\end{proposition}

\begin{proofof}
It is clear from the definition of $\Gamma$ that the number of common transversals for the pair $(H, K)$ in $G$ equals the product over the number of common transversals for the $t_\sigma$ left-$H$ and right-$K$ cosets in each component $\Delta_\sigma$ of $\Gamma$. Moreover, the fact that each component $\Delta_\sigma$ is a symmetric complete bipartite graph $K_{t_\sigma, t_\sigma}$, while the weight $w_\sigma$ on every edge of $\Delta_\sigma$ is the size of the intersection of any left-$H$ with with any right-$K$ coset of
$\Delta_\sigma$, implies that the number of common $(H, K)$-double transversals for the $t_\sigma$ left-$H$ and $t_\sigma$ right-$K$ cosets of $\Delta_\sigma$ equals $t_\sigma!\cdot w_\sigma^{t_\sigma}$. By Lemma~\ref{Lem:HKWeight}, we have $w_\sigma = \vert H\vert/t_\sigma$, in particular, $t_\sigma \,\big\vert\, \vert H\vert$ for each $\sigma \in [s]$. Since $\sum_\sigma t_\sigma = n$, it follows that 
\begin{equation*}
\vert \mathrm{DT}_G(H, K)\vert = \prod_{\sigma = 1}^s t_\sigma!\, w_\sigma^{t_\sigma}
 = \prod_{\sigma = 1}^s \left(\frac{t_\sigma!}{t_\sigma^{t_\sigma}}\cdot \vert H\vert ^{t_\sigma}\right)
 = \vert H \vert^{\sum_\sigma t_\sigma}\cdot \prod_{\sigma = 1}^s \frac{t_\sigma!}{t_\sigma^{t_\sigma}} 
 = \vert H\vert^n\cdot \prod_{\sigma = 1}^s \frac{t_\sigma!}{t_\sigma^{t_\sigma}}.
\end{equation*} 
Hence, 
\[
P_G(H, K) = \frac{\vert \mathrm{DT}_G(H, K)\vert}{\vert H\vert^n} = \prod_{\sigma = 1}^s \frac{t_\sigma!}{t_\sigma^{t_\sigma}},
\]
as claimed.
\end{proofof}

\begin{definition}
Let $H \leq G$ and let $\{\Delta_{t_{\sigma}}\}_\sigma$ be the family of connected components of $\Gamma^G_{H},$ where the family of positive integers $\{ t_{\sigma} \}$ is defined as above.
We list the integers $t_{\sigma} $ in decreasing order, say $( t_1, \cdots , t_s),$ and call the resulting vector the $t$-vector of $H$ in $G$.
\end{definition}

\begin{corollary}
\label{Cor:P = 1}
We have $P_G(H, K) = 1$ if, and only if, $H = K \unlhd G$. Hence  $\mathrm{tp}(G) = 1$ if, and only if, $G$ is a Dedekind group; that is, if, and only if, every subgroup in $G$ is normal.
\end{corollary}

\begin{proofof}
If $H = K \unlhd G$, then $P_G(H, K) = P_G(H) = 1$. Conversely, suppose that $H, K \leq G$ are such that $P_G(H, K) = 1$. By Proposition~\ref{Prop:PFormula},
\[
P_G(H, K) = \prod_{\sigma = 1}^s \frac{t_\sigma!}{t_\sigma^{t_\sigma}} = 1, 
\]
where $s = \vert K\backslash G/H\vert$, so that $t_\sigma = 1$ for all $\sigma \in [s]$. Since $\sum_\sigma t_\sigma = n$, it follows that $m = s = n$, thus, by \eqref{Eq:mForm},
$H$ and $K$ are conjugate, and we have $N_G(H) = G$, so that $H \unlhd G$ and $H = K$.
The rest of the Corollary follows directly from the definition of $\mathrm{tp}(G)$.
\end{proofof}

\begin{corollary}
\label{Cor:FrobGen}
Suppose that $H, K \leq G$ are such that $(G:H) = n = (G:K),$ and that $H \cap K^g = 1$ for $g\not\in KH$. Then
\[
P_G(H, K) = \left(\frac{\vert H\vert!}{\vert H\vert^{\vert H\vert}}\right)^{\frac{n-m}{\vert H\vert}}.
\]
\end{corollary}

\begin{proofof}
Apart from the double coset $HK$ for every other $g_{\sigma} \in \Delta_{\sigma}$  we have 
\[
t_\sigma = \frac{\vert H\vert}{\vert H \cap K^{g_\sigma}\vert} = \vert H\vert,\quad g_\sigma \not\in KH
\]
by Lemma~\ref{Lem:HKWeight}(b). Hence  
\[
P_G(H, K) = \left(\frac{\vert H\vert!}{\vert H\vert^{\vert H\vert}}\right)^{s-1} = \left(\frac{\vert H\vert!}{\vert H\vert^{\vert H\vert}}\right)^{\frac{n-1}{\vert H\vert}} 
\]
by Proposition~\ref{Prop:PFormula}, plus the fact that $n=\sum_{\sigma =1}^s t_\sigma = 1+|H|(s-1)$.
\end{proofof}

\begin{corollary}
\label{Cor:Frob}
Let $G$ be a finite Frobenius group with Frobenius complement $H,$ and let $(G:H) = n$. Then 
\[
P_G(H) = \left(\frac{\vert H\vert!}{\vert H\vert^{\vert H\vert}}\right)^{\frac{n-1}{\vert H\vert}}.
\]
\end{corollary}
\begin{proofof}
Here $H = K$, $H \cap H^g = 1$ for $g\not\in H$, and $N_G(H) = H$, whence the result by Corollary~\ref{Cor:FrobGen}.
\end{proofof}

\begin{corollary}
\label{Cor:Hprime}
Let $H\leq G$ be a subgroup of order $\vert H\vert = p$ a prime, and suppose that $(G:H) = n$. Then
\[
P_G(H) = \left(\frac{p!}{p^p}\right)^{\frac{n-m}{p}},
\]
where $m = (N_G(H):H),$ and the exponent on the right-hand side is an integer.
\end{corollary}

\begin{proofof}
Apart from $m$ trivial components in 
$\Gamma^G_H \coloneqq \Gamma^G_{H, H}$, 
we have $\frac{n-m}{p}$ components $\Delta_\sigma \cong K_{p, p}$. In particular, this quotient is an integer. The result now follows from Proposition~\ref{Prop:PFormula}. 
\end{proofof}

\subsection{The function $P_G$ and permanents}

\noindent At this point, we wish to briefly discuss a connection between the probability $P_G(H, K)$, where $H$ and $K$ are subgroups of the finite group $G$ such that $(G:H) = n = (G:K)$, and the permanent of a certain associated matrix (for properties of permanents used in what follows, the reader is referred to the standard reference \cite{Minc}). The matrix we associate to the triple $(G,H,K)$ is the $n \times n$ weight matrix $W = W_{H,K}^G = (w_{i,j})$, whose $(i,j)$ entry is 
\[
w_{i,j} = |l_iH \cap K r_j|.
\]
That is, $w_{i,j}$ is the weight assigned to the edge $e = l_iH - K r_j $ of the coset intersection graph $\Gamma = \Gamma^G_{H, K}$. 
A simple combinatorial argument shows that 
\[
\vert\mathrm{DT}_G(H, K)\vert = \sum_{\sigma\in S_{n}}\,\prod_{i = 1}^n w_{i,\sigma(i)}. \
\]
As the right hand side in the above equation is exactly the permanent of the matrix $W$,
we get 
\begin{equation}
 |\mathrm{DT}_G(H, K)| = \mathrm{per}(W_{H,K}^G).
\end{equation}

\noindent Observe now that all row and column sums of $W$ are equal to $|H| = |K|$; that is, 
\[
\sum_{i = 1}^n w_{i,r} = |H| = \sum_{j = 1}^n w_{t,j},\quad (1\leq r, t\leq n).\\
\]
Thus, the matrix 
\[
M = M_{H,K}^G \coloneqq \frac{1}{|H|} \cdot W_{H,K}^G
\]
is an $n \times n$ doubly stochastic matrix
 (all row and column sums equal to 1); 
furthermore, its permanent is exactly the common transversal
probability for the subgroups $ H, K \leq G$, as 
\[
\mathrm{per}(M) = \frac{1}{|H|^n} \cdot \mathrm{per}(W_{H,K}^G) = P_G(H,K). 
\]
In addition, after appropriate row and column permutations, we can group together cosets (left for $H$ and right for $K$) which intersect. In this way, $M$ is seen to be equivalent to a block diagonal matrix: there exist $n \times n$ permutation matrices $P$ and $Q$, such that 
\begin{equation}
 P M Q = \begin{bmatrix} 
 M_1 & 0  & \cdots & 0  \\
 0  & M_2 & \cdots & 0  \\
 \vdots & \vdots & \ddots & \vdots \\
 0  & 0  & \cdots & M_s
 \end{bmatrix}
\,\,,
\end{equation}
 where, for $\sigma = 1, \ldots, s$, the matrix $M_\sigma$ is the $t_\sigma \times t_\sigma$ rational matrix
whose entries all equal $\frac{1}{t_\sigma}$. As usual, $t_{\sigma} = (H : H \cap K^{g_{\sigma}}) $, where 
 $\{t_{\sigma}\}_{\sigma = 1}^{s}$ is the family of sizes of the connected components of the coset intersection graph $\Gamma^G_{H, K}$.
 Collecting together these observations, we get most of the following. 
\begin{proposition}
The doubly stochastic matrix $M = M_{H,K}^G$ satisfies 
\begin{equation}
\label{Eq:MPerm}
\mathrm{per}(M) = P_G(H, K) = \prod_{\sigma = 1}^s\frac{(t_\sigma)!}{t_\sigma^{t_\sigma}}.
\end{equation}
The eigenvalues of $M$ are $0$ 
with multiplicity $n - s,$ and $1$ with multiplicity $s$. Moreover, $M$ is positive semi-definite.
\end{proposition} 

\begin{proofof}
We have already seen that $ \mathrm{per}(M) = P_G(H, K)$ and the right-hand side equals the desired product according to Proposition \ref{Prop:PFormula}. We remark here that an alternative way to see that $\mathrm{per}(M_{H,K}^G) = \prod_{\sigma = 1}^s\frac{(t_\sigma)!}{t_\sigma^{t_\sigma}}$ is the following: 
Since $PMQ = \mathrm{diag}(M_1, \ldots, M_s)$,
we have $\mathrm{per}(PMQ) = \mathrm{per}(M) = \prod_{\sigma = 1}^{s} \mathrm{per}(M_\sigma)$. On the other hand, working directly from the definition, the permanent of a square matrix $M = (m_{i, j})$ of size $c$
with all entries equal to some fixed number $a$ is seen to be 
\[
\mathrm{per}(M) = \sum_{\sigma\in S_{c}}\,\prod_{i = 1}^c m_{i,\sigma(i)} = 
\sum_{\sigma\in S_{c}} a^c = c!\,a^c.
\]
Specifying $a = 1/t_\sigma$ and $c = t_\sigma$, for each $\sigma \in [s]$, yields 
\[
\mathrm{per}(M_\sigma) = t_\sigma!/t_\sigma^{t_\sigma},
\] 
whence \eqref{Eq:MPerm}. The assertion concerning the eigenvalues of $M$ follows from the fact that the characteristic polynomial
of the matrix $M_\sigma$ is 
\[
\mathrm{det}(M_\sigma -\lambda I) = (-\lambda)^{t_\sigma-1}(1-\lambda),
\]
which implies that the characteristic polynomial
of the matrix $M$ is given by 
\[
\mathrm{det}(M -\lambda I) = (-\lambda)^{n-s}(1-\lambda)^s.
\]
Thus, $\mathrm{tr}(M) = s$, while $\mathrm{det}(M) = 0$, which can, of course, also be seen directly. 
Also, $\mathrm{rk}(M) = s$. 
Since $M$ is symmetric and its eigenvalues
are non-negative, it follows that $M$ is positive semi-definite.
\end{proofof}

\subsection{A limit theorem}\label{Sec:PLim}
\noindent
The aim in this subsection is twofold.
Firstly, we will give sharp bounds for the value $P_G(H, K)$;
so, for example, we will prove that for non-normal conjugate subgroups its value cannot exceed $1/2$. 
Secondly,  we will show that the bigger the index $n = (G : H) = (G : K)$ is
the smaller the value $P_G(H, K)$ becomes. 
We start with the following purely arithmetic result.

\begin{lemma}
\label{Lem:ProdEst}
Let $s\geq 1,$ and let $\{t_\sigma\}_{\sigma = 1}^s$ 
be a family of positive integers with $\sum_{\sigma = 1}^s t_\sigma = n$. 
Then 
\begin{equation}
\label{Eq:PAsymp1}
\frac{n!}{n^n} \leq \prod_{\sigma = 1}^s \frac{t_\sigma!}{t_\sigma^{t_\sigma}}\, \leq\, \left(\frac{n+s}{2n}\right)^n.
\end{equation}
Furthermore if $m$ of the $t_\sigma$ are equal to {\em 1,} then 
\begin{equation}
 \label{Eq:PAsymp2}
\frac{(n-m)!}{(n-m)^{n-m}} \leq \prod_{\sigma = 1}^s \frac{t_\sigma!}{t_\sigma^{t_\sigma}} \leq \frac{1}{2^{s-m}}.
\end{equation}
\end{lemma}

\begin{proofof}
We first prove the left-hand inequality in~\eqref{Eq:PAsymp1} by induction on $s$. 
For $s = 1$, both sides are equal. 
Next, suppose that $s = 2$, so that $t_1 + t_2 = n$. 
Then $t_1 ! \cdot t_2 ! = \binom{n}{t_1}^{-1} n!$,
while $n^n = (t_1 +t_2)^n \geq \binom{n}{t_1} t_1^{t_1} \cdot t_2 ^{t_2}$. 
Thus, 
\[
\frac{n!}{n^n} = \frac{\binom{n}{t_1} t_1! \cdot t_2!}{n^n } \leq \frac{\binom{n}{t_1} t_1! \cdot t_2!}{\binom{n}{t_1} t_1^{t_1} \cdot t_2 ^{t_2} } = \prod_{\sigma = 1}^2 \frac{t_\sigma!}{t_\sigma^{t_\sigma}},
\]
as desired. 
For the induction step, suppose that the left-hand inequality in~\eqref{Eq:PAsymp1} holds for all $s\leq u$ with some $u\geq2$, and let $n = t_1 + t_2 + \cdots + t_{u+1}$. Then $n = t_1 + (n - t_1)$, and the inductive hypothesis for $s = 2$ and $s = u$ gives 
\[
\frac{n!}{n^n} \leq \frac{t_1!}{t_1^{t_1}} \cdot \frac{(n-t_1)!}{(n-t_1)^{n-t_1}} \leq \frac{t_1!}{t_1^{t_1}} \cdot \prod_{\sigma = 2}^{u+1} \frac{t_\sigma!}{t_\sigma^{t_\sigma}}
 = \prod_{\sigma = 1}^{u+1} \frac{t_\sigma!}{t_\sigma^{t_\sigma}},
\]
whence our result.

\noindent The right-hand inequality in \eqref{Eq:PAsymp1} follows from the arithmetic-geometric mean inequality as follows:
\begin{multline*}
\prod_{\sigma = 1}^s \frac{t_\sigma!}{t_\sigma^{t_\sigma}} = 
\prod_{\sigma = 1}^s\, \prod_{\tau = 1}^{t_\sigma} \tau/t_\sigma 
\leq\,
\left(\sum_{\sigma = 1}^{s}\, \sum_{\tau = 1}^{t_\sigma}\,\tau/t_\sigma\Big/\sum_{\sigma = 1}^s t_\sigma\right)^{\sum_{\sigma = 1}^{s}t_\sigma} 
 = \left(\frac{1}{n}\,\sum_{\sigma = 1}^s \frac{1}{t_\sigma} \frac{t_\sigma (t_\sigma + 1)}{2}\right)^n\\[1mm]
 = \left(\sum_{\sigma = 1}^s (t_\sigma + 1)\,\Big/\,2n\right)^n 
 = \left(\frac{n+s}{2n}\right)^n.
\end{multline*}
This establishes the first part of the lemma. Now suppose that 
\[
t_1 = t_2 = \cdots = t_m = 1
\] 
for some $m \leq s$, while $t_\sigma \geq 2$ for $m+1 \leq \sigma \leq s$. Then $n-m = t_{m+1} + \cdots + t_s$ thus, by the left-hand inequality in \eqref{Eq:PAsymp1}, 
\[
\frac{(n-m)!}{(n-m)^{n-m}} \leq 
\prod_{\sigma = m+1}^s \frac{t_\sigma!}{t_\sigma^{t_\sigma}} = 
\prod_{\sigma = 1}^s \frac{t_\sigma!}{t_\sigma^{t_\sigma}}.
\]
Furthermore, for each $\sigma > m$ we have $t_\sigma \geq 2$, so that $t_\sigma! / t_\sigma^{t_\sigma} \leq 1/2$, thus 
\[
\prod_{\sigma = 1}^s t_\sigma! / t_\sigma^{t_\sigma} \leq 2^{-(s-m)}.
\]
This completes the proof of the lemma.
\end{proofof}

\begin{corollary}
\label{Cor:Prod}
Let $G$ be a finite group, and let $H, K\leq G$ be conjugate, non-normal subgroups of index $n$ in $G$. Then we have
\begin{equation}
\label{Eq:PEst}
\frac{(n-1)!}{(n-1)^{n-1}} \leq P_G(H, K) \leq \frac{1}{2},
\end{equation}
and these bounds are best possible. In particular, for $n = 3,$ we have $P_G(H, K) = 1/2$. 
\end{corollary}

\begin{proofof}
Let $H, K$ be conjugate subgroups of $G$. Then, according to Lemma~\ref{Lem:HKWeight}(a), the number $m$ of connected components $\Delta_\sigma $  of $\Gamma^G_{H, K}$ of type $\Delta_\sigma \cong K_{1, 1}$ is at least one, i.e.,   $m\geq 1$. On the other hand, we cannot have $m = s = \vert K\backslash G/H\vert$, since otherwise $P_G(H, K) = 1$ by Proposition~\ref{Prop:PFormula}, so that $H = K \unlhd G$ by Corollary~\ref{Cor:P = 1}, contradicting our hypothesis. Hence, we have $1\leq m<s$. Since the left-hand side of \eqref{Eq:PAsymp2} decreases with $m$, while the right-hand side increases with $m$, the result follows.

\noindent We next focus on the sharpness of the bounds in \eqref{Eq:PEst}. Let $F$ be a field with $q$ elements, 
where $q$ is a power of a prime. 
Then the action of the multiplicative group $F^{\ast}$ 
on the additive group of $F$ is of Frobenius type. 
The resulting group $G = F \rtimes F^\times$ is a Frobenius group with Frobenius complement $H \cong F^{\ast}$ of order $q - 1$ and index $q$. 
Thus, by Corollary~\ref{Cor:Frob}, 
\[
P_G(H) = \left[\frac{(q - 1)!}{(q - 1)^{q - 1}}\right]^{\frac{q - 1}{q - 1}} = 
\frac{(q - 1)!}{(q - 1)^{q - 1}},
\]
which shows that the lower bound in Equation~\eqref{Eq:PEst} is attained.

\noindent Also, for $q = 3$, let $G \cong S_3$, and let $H$ be any Sylow $2$-subgroup of $G$. In this case, $P_G(H) = 1/2$, so that the upper bound in \eqref{Eq:PEst} is attained as well.
\end{proofof}

\noindent We can now establish our first main result.

\begin{theorem}\label{Thm:Main}
Let $G$ be a finite group, and let $H, K\leq G$ be subgroups such that \linebreak $(G:H) = n = (G:K),$ and such that at least one of $H, K$ is not normal in $G$. Then 
\[
\lim_{n \to \infty} P_G(H,K) = 0.
\]
\end{theorem}

\begin{proofof}
According to Proposition~\ref{Prop:PFormula} and Lemma~\ref{Lem:ProdEst}, we have 
\begin{equation}\label{eq1}
 P_G(H,K) = \prod_{i = 1}^s \frac{t_i!}{t_i^{t_i}} \leq \left(\frac{n+s}{2n}\right)^n, 
\end{equation}
where $s = \vert K\backslash G/H\vert$ equals the number of connected components of $\Gamma^G_{H, K}$. 
According to Lemma~\ref{Lem:HKWeight} exactly $m$ of those $t_\sigma$'s are equal to 1, while for the rest we get $t_\sigma \geq 2$, where $m= (N_G(H):H)$ if $H, K$ are conjugate,  or $m=0$ if they are not.
As $\sum_{\sigma = 1}^s t_\sigma = n$, we conclude that 
\[
s \leq m + \frac{n-m}{2} = \frac{n+m}{2}.
\]
Furthermore, by \eqref{Eq:mForm}, $m$ is either zero, or divides $n$ \textit{properly} (otherwise we would have $H = K \unlhd G$ contradicting our hypothesis); thus $m \leq \frac{n}{2}$. Hence, $s \leq \frac{3n}{4}$ in either case, and Equation~\eqref{eq1} implies 
\begin{equation}
\label{Eq:PBound}
 P_G(H,K) \leq \left(\frac{n+s}{2n}\right)^n \leq \left(\frac{7}{8}\right)^n.
\end{equation}

As $\lim_{n \to \infty} \left(\frac{7}{8}\right)^n = 0$, the theorem follows.
\end{proofof}

\subsection{An improved bound for $P_G(H, K)$}
\label{Sec:PImproved Bound}

\noindent We digress briefly to discuss an upper bound for the function $\prod_{\sigma = 1}^s t_\sigma!/t_\sigma^{t_\sigma}$, which is stronger than \eqref{Eq:PBound}. For $x>0$, let 
\[
f(x) \coloneqq \frac{\Gamma(x+1)}{x^x},
\]
where $\Gamma(x) = \int_{0}^{\infty} t^{x-1} e^{-t} d t$
is the well-known $\Gamma$--function.
Observe that $f(t_\sigma) = t_\sigma! / t_\sigma^{t_\sigma}$
and thus
\[
\prod_{\sigma = 1}^s \frac{t_\sigma!}{t_\sigma^{t_\sigma}} = 
\prod_{\sigma = 1}^s f(t_{\sigma})\,.
\]
We argue that $f$ is strictly log-concave; that is, the function $g(x) = \log f(x)$ is strictly concave. 
First, $g$ is infinitely differentiable.
Second, 
\[
g''(x) = \psi'(x+1) \,-\, \frac{1}{x},
\]
where $\psi(x) = (\log\, \Gamma(x))' = \Gamma'(x) / \Gamma(x)$ is the so-called digamma function. By \cite[Lemma~2]{GQ}, we have
\[
\psi'(x+1) < e^{\frac{1}{x+1}} -1, \quad (x >0),
\]
while
\[
e^{\frac{1}{x+1}} -1 < \frac{1}{x},\quad (x>0)
\]
by virtue of the elementary inequality
\[
e < \left(1+\frac{1}{x}\right)^{x+1}, 
\]
which is valid for all $x > 0$. Thus
\[
g''(x) = \psi'(x+1) - \frac{1}{x} < 0, \quad (x > 0),
\]
which implies that $g$ is strictly concave, as claimed.

\noindent Now let 
\[
\Pi \coloneqq \prod_{\sigma = 1}^s f(t_{\sigma})\,.
\]
Then 
\[
\log\, \Pi = \sum_{\sigma = 1}^{s} \log f(t_{\sigma})
 \leq 
s \cdot \log\, f\left(\frac{\sum_{\sigma = 1}^{s}t_{\sigma}}{s}\right) = 
s \cdot \log\, f\left(\frac{n}{s}\right)
\]
where the inequality is a consequence of Jensen's inequality applied to the concave function $g(x) = \log f(x)$. We conclude that 
\[
\prod_{\sigma = 1}^s \frac{t_\sigma!}{t_\sigma^{t_\sigma}} = 
\prod_{\sigma = 1}^s f(t_{\sigma}) \leq 
f\left(\frac{n}{s}\right)^{s}.
\]
At this point we are interested in comparing the two bounds
\[
f\left(\frac{n}{s}\right)^s \quad \mbox{and} \quad
\left(\frac{n+s}{2n}\right)^n\,.
\]
Specifically, we will show the following.

\begin{proposition}\label{Prop:0.97}
For all $n \in \mathbb{N}$ 
and for all $s$ with $s \leq \frac{3}{4}n$, we have
\begin{equation}\label{Eq:0.97}
f\left(\frac{n}{s}\right)^s \leq
c^n\left(\frac{n+s}{2n}\right)^n\,,
\end{equation}
where 
\[
c \coloneqq \frac{8}{7} f(4/3)^{3/4} = 0.976986... \,.
\]
\end{proposition}
\begin{proofof}
Proving the above inequality is tantamount
to proving that $f(y) \leq c^y \left(\frac{y+1}{2y}\right)^y$
for all $y \geq 4/3$,
where we have made the substitution $y \coloneqq n/s$.
It will clearly suffice to show that
$\log f(y) \leq 
y \left[\log\left(\frac{y+1}{2y}\right) +\log c\right]$.
So let 
\begin{align*}
H(y) &\,\coloneqq\, \log\, f(y) - y \left[\log\left(\frac{y+1}{2y}\right) +\log c\right] \\[1.5mm]
 & = \log\, \Gamma(y+1) \,-\, y \,\log\, y \,-\, y \left[\log\left(\frac{y+1}{2y}\right) +\log c\right] \\[1.5mm]
 & = \log\, \Gamma(y+1) \,-\, y \left[\log\left(\frac{y+1}{2}\right) +\log c\right]
\end{align*}
for $y \geq 4/3$, and observe that $H(4/3) = 0$.
Thus, the claim will have been established,
provided that we can show that $H'(y) \leq 0$ 
for $y \geq 4/3$. Now observe that
\[
H'(y) = \psi(y+1) - \log\left(\frac{y+1}{2}\right) - \log c - \frac{y}{y+1}\,.
\]
By Lemma 1 in \cite{GQ}, we have 
\[
\psi(y+1) < \log(y+1) \,-\, \frac{1}{2(y+1)}\quad (y > 0).
\]
Hence, 
\begin{align}
H'(y) \,&<\, \log(y+1) \,-\, \frac{1}{2(y+1)} \,-\, \log\left(\frac{y+1}{2}\right) \,-\,
\log c -\frac{y}{y+1} \\
 \,& = \, 
\frac{y\left[2 \log (2/c) - 2\right] + 2 \log (2/c) - 1}{2(y+1)}.
\end{align}
We note that the numerator $N$ in the last expression is negative, since $N\geq 0$ would imply that
\[
y \leq \frac{2 \log (2/c) - 1}{2-2 \log (2/c)} \,\approx\, 
0.763233 < \frac{4}{3} \, ,
\]
contradicting our assumption that $y\geq 4/3$, while the denominator is positive.
Thus $H'(y) \leq 0$, which is what we wanted to prove.
\end{proofof}

\begin{remark}
It is clear that we have equality in~\eqref{Eq:0.97}
if and only if $n = (4/3)s$ since $H(4/3) = 0$
and $H'(y)$ is in fact strictly negative for $y > 4/3$.
Moreover, if $s$ remains bounded and thus $y = n/s$ grows without bound, $f(n/s)^s$
becomes an even better bound than $\left(\frac{n+s}{2n}\right)^n$ since
\[
f(n/s)^s \sim \frac{\left(2 \pi \frac{n}{s}\right)^{s/2}}{e^n}
\quad\quad \text{while}
\quad\quad
\left(\frac{n+s}{2n}\right)^n \sim \frac{e^s}{2^n},
\]
where the first relation follows from Stirling's 
approximation.
\end{remark}


\subsection{Subgroups, homomorphic images, and the function $P_G(-)$}
\label{Sec:SubHomP}
\noindent From now on, we shall concentrate on the case where $H = K$; that is, we shall focus on the functions of one variable $P_G(-)$ and $\mathrm{tp}(-)$. Our aim in this subsection is to  show that $P_G(-)$ behaves well with respect to subgroups and homomorphic images. 

\begin{proposition}
\label{Prop:PGSub}
Let $K \leq H \leq G$ be finite groups. Then every connected component of $\Gamma^H_K$ is also a connected component of $\Gamma^G_K;$ in particular, $\Gamma^H_K$ is an induced subgraph of $\Gamma^G_K,$ and we have 
\begin{equation}
\label{Eq:PGSub}
P_G(K) \leq P_H(K).
\end{equation}
\end{proposition}

\begin{proofof}
First, it is clear that $\Gamma^H_K$ is a subgraph of $\Gamma^G_K$. Second, we observe that
\[
hK \cap Kg\,\neq\,\emptyset\,\Longrightarrow\, g\in H,\quad (h\in H,\, g\in G),
\]
since, by hypothesis, there have to exist elements $k_1, k_2\in K$ such that $h k_1 = k_2g$, so that $g = k_2^{-1} h k_1\in H$, as $K\leq H$. Therefore, an edge in $\Gamma^G_K$ with one vertex in the set $H/K$ necessarily has its other bounding vertex in $K\backslash H$, which implies that the connected components of $\Gamma^H_K$ are \emph{identical} with certain components of $\Gamma^G_K$. Third, we obviously have
\[
s_H \coloneqq \vert K\backslash H/K\vert \leq \vert K\backslash G/K\vert = : s_G.
\]
Thinking of the connected components of $\Gamma^H_K$ as being listed first among the components of $\Gamma^G_K$, we now find that 
\begin{equation*}
 P_G(K) = \prod_{\sigma = 1}^{s_G} \frac{(t_\sigma)!}{t_\sigma^{t_\sigma}}
 = 
 \prod_{\sigma = 1}^{s_H}\frac{(t_\sigma)!}{t_\sigma^{t_\sigma}} \times \prod_{\sigma = s_H + 1}^{s_G}\frac{(t_\sigma)!}{t_\sigma^{t_\sigma}} 
\leq \prod_{\sigma = 1}^{s_H}\frac{(t_\sigma)!}{t_\sigma^{t_\sigma}} = P_H(K),
\end{equation*}
whence \eqref{Eq:PGSub}. 
\end{proofof}

\noindent Our next result concerns the connection of $P_G$ with homomorphic images. 

\begin{theorem}
\label{Thm:PHom}
Let $G$ and $K$ be finite groups, $f: G \rightarrow K$ a group homomorphism, and let $H \leq G$ be a subgroup. Then, if $N \coloneqq \mathrm{ker}(f) \leq H,$ we have
\[
P_G(H) = P_{f(G)}(f(H)).
\]
\end{theorem}

\begin{proofof}
Our aim is to define and analyse a certain map 
\[
\Phi: \mathrm{DT}_G(H) \coloneqq \mathrm{DT}_G(H, H)\,\longrightarrow\, \mathrm{DT}_{f(G)}(f(H)).
\]
Suppose that $(G:H) = n$, and let $T = \{t_1, t_2, \ldots, t_n\}$ be a two-sided transversal for $H$ in $G$. Applying the homomorphism $f$, we obtain what looks at first like a multiset $\overline{T} = \{f(t_1), f(t_2), \ldots, f(t_n)\} \subseteq K$, and we claim that $\overline{T}$ is in fact a two-sided transversal for $f(H)$ in $f(G)$; in particular, $\overline{T}$ is an $n$-set. Suppose first that 
\[
f(t_{\nu_1} H) \cap f(t_{\nu_2} H) = f(t_{\nu_1}) f(H) \cap f(t_{\nu_2}) f(H) \,\neq\, \emptyset, \quad (\nu_1, \nu_2 \in [n]).
\]
Then there exist elements $h_1, h_2 \in H$, such that $f(t_{\nu_1} h_1) = f(t_{\nu_2} h_2)$, or $f(t_{\nu_1} h_1 h_2^{-1} t_{\nu_2}^{-1}) = 1$. Thus, there are elements $n, n'\in N$, such that 
\[
t_{\nu_2} = n^{-1} t_{\nu_1} h_1 h_2^{-1} = t_{\nu_1} n' h_1 h_2^{-1} \,\in t_{\nu_1} H,
\]
as $N \leq H$ by hypothesis. This forces $\nu_1 = \nu_2$, since $T$ is a left transversal for $H$ in $G$. Since $(f(G): f(H)) \leq n$, it follows that $\overline{T}$ is a complete set of pairwise inequivalent representatives for the left cosets of $f(H)$ in $f(G)$. Similarly, if 
\[
f(H t_{\nu_1}) \cap f(H t_{\nu_2}) = f(H) f(t_{\nu_1}) \cap f(H) f(t_{\nu_2})\,\neq\, \emptyset,\quad (\nu_1, \nu_2 \in [n]),
\]
we find that $f(h_1 t_{\nu_1} t_{\nu_2}^{-1} h_2^{-1}) = 1$ for some $h_1, h_2\in H$, or 
\[
t_{\nu_1} t_{\nu_2}^{-1} = h_1^{-1} n h_2 \in H,\quad (n\in N),
\]
again implying $\nu_1 = \nu_2$, so that $\overline{T}$ is also a right transversal for $f(H)$ in $f(G)$; therefore, $\overline{T} \in \mathrm{DT}_{f(G)}(f(H))$.

\noindent We now set $\Phi(T) \coloneqq \overline{T}$, to obtain a well-defined map from $\mathrm{DT}_G(H)$ to $\mathrm{DT}_{f(G)}(f(H))$, and claim that $\Phi$ is surjective. In order to justify this claim, let $\{f(s_\nu)\}_{1\leq \nu\leq n} \in \mathrm{DT}_{f(G)}(f(H))$, pick arbitrary pre-images $\widetilde{s}_\nu \in f^{-1}(f(s_\nu))$ for $1\leq \nu\leq n$, and form the set $\widetilde{S} = $ \linebreak $\{\widetilde{s}_\nu: 1\leq \nu\leq n\}$. Then $\widetilde{S} \in \mathrm{DT}_G(H)$. Indeed, if $\widetilde{s}_{\nu_1}^{-1} \widetilde{s}_{\nu_2} \in H$, then
\[
f(\widetilde{s}_{\nu_1}^{-1} \widetilde{s}_{\nu_2}) = f(\widetilde{s}_{\nu_1})^{-1} f(\widetilde{s}_{\nu_2}) = f(s_{\nu_1})^{-1} f(s_{\nu_2}) \,\in\, f(H),
\]
so that $\nu_1 = \nu_2$. A similar argument works for right cosets and, by construction, 
\[
\Phi(\widetilde{S}) = \left\{f(\widetilde{s}_\nu)\right\}_{1\leq \nu\leq n} = \left\{f(s_\nu)\right\}_{1\leq \nu \leq n},
\]
whence our claim. 

\noindent We observe that, as $N = \mathrm{ker}(f) \leq H$, if $T = \{t_\nu\}_{1\leq \nu\leq n} \in \mathrm{DT}_G(H)$, so is $T' = \{m_\nu t_\nu\}_{1\leq \nu\leq n}$, where $m_1, \ldots, m_n\in N$, as $H m_\nu t_\nu = H t_\nu$ and $m_\nu t_\nu H = t_\nu m_\nu' H = t_\nu H$, where $m_\nu' \in N$. Moreover, we have $\Phi(T) = \Phi(T')$, since $f(m_\nu t_\nu) = f(t_\nu)$. Consequently, for each $\overline{T}\in \mathrm{DT}_{f(G)}(f(H))$, there are at least $\vert N\vert^n$ distinct pre-images under $\Phi$. Actually, we have 
\[
\vert \Phi^{-1}(\overline{T})\vert = \vert N\vert^n,\quad (\overline{T}\in \mathrm{DT}_{f(G)}(f(H))).
\]
Namely, if $S = \{s_\nu\}_{1\leq \nu\leq n} \in \mathrm{DT}_G(H)$ with $\Phi(S) = \Phi(T)$, where $T \in \mathrm{DT}_G(H)$ is a given two-sided transversal then, after appropriate rearrangement of indices, we have $f(s_\nu) = f(t_\nu)$ for $1\leq \nu\leq n$, thus
\[
s_\nu = m_\nu t_\nu,\quad (1\leq \nu\leq n),
\]
where $m_\nu\in N$, as claimed. Summarising our findings so far, we have shown that 
\[
\big\vert \mathrm{DT}_{f(G)}(f(H))\big\vert = \frac{\vert \mathrm{DT}_G(H)\vert}{\vert N\vert^n}.
\]
It follows that
\begin{equation*}
P_{f(G)}(f(H)) = \frac{\vert \mathrm{DT}_{f(G)}(f(H))\vert}{\vert f(H)\vert^n} = \frac{\vert \mathrm{DT}_G(H)\vert/\vert N\vert^n}{\vert H/N\vert^n} 
 = \frac{\vert \mathrm{DT}_G(H)\vert}{\vert H\vert^n} = P_G(H),
\end{equation*}
completing the proof.
\end{proofof}

\begin{corollary}\label{cor:iso}
Let $G$ be a finite group, let $H$ be a subgroup of $G,$ and let $f: G \rightarrow K$ be an isomorphism.
Then we have $P_G(H) = P_K(f(H)),$ 
in particular, $\mathrm{tp}(G) = \mathrm{tp}(K);$ that is, $\mathrm{tp}(-)$ is an isomorphism invariant.
\end{corollary}

\noindent Taking $f: G \rightarrow G/N$ the canonical projection, where $N$ is a normal subgroup of $G$, Theorem~\ref{Thm:PHom} gives the following.

\begin{corollary}\label{cor:normal}
Let $G$ be a finite group, and let $N \leq H$ 
be subgroups of $G$ with $N$ normal in $G$. 
Then $P_G(H) = P_{G/N}(H/N)$.
\end{corollary}

\section{The common transversal probability $\mathrm{tp}(-)$ of a group.}\label{Sec:tp}
We now turn our attention to the invariant  $\mathrm{tp}(-)$, defined as  \[
\mathrm{tp}(G) \coloneqq \min_{H\leq G}\, P_G(H).
\]
The main objective of the section is to relate the values of $\mathrm{tp}(G)$ to key properties for the group $G$. Roughly  this means that the larger $\mathrm{tp}(G)$ is, the more normal structure $G$ exhibits, in particular, we show that solubility, supersolubility and nilpotency are characterised by certain values of $\mathrm{tp}(G)$. We also compute the common transversal probability of certain groups and
show that specific values of $\mathrm{tp}$  are achieved by specific groups while others are not attained at all.  We also show that $\mathrm{tp}$ behaves well with respect to subgroups, quotients and group extensions.

\subsection{The function $\mathrm{tp}(-)$ and group structure}\label{Sec:GrpStructure}
\noindent Our first result  records the behaviour of $\mathrm{tp}(-)$ under taking subgroups, quotients, and sections. 

\begin{proposition}\label{Prop:structproperties}
Let $G$ be a finite group.
\vspace{-2mm}
\begin{enumerate}
\item[(i)] For every proper subgroup $H$ of $G$ we have $\mathrm{tp}(G) \leq \mathrm{tp}(H)$.
\vspace{1mm}
\item[(ii)] For every normal subgroup $N$ of $G$ we have $\mathrm{tp}(G) \leq \mathrm{tp}(G/N)$.
\vspace{1mm}
\item[(iii)] For every section $X$ of $G$ we have $\mathrm{tp}(G) \leq \mathrm{tp}(X)$.
\vspace{1mm}
\item[(iv)] If $G$ is a non-Dedekind $p$-group, then $\mathrm{tp}(G) \leq p!/p^p$.
\end{enumerate}
\end{proposition}

\begin{proofof}
(i) Let $ K \leq H $ be such that $\mathrm{tp}(H) = P_H(K)$. Then, according to 
Proposition~\ref{Prop:PGSub}, we have $P_G(K) \leq P_H(K)$. Hence, 
\[
\mathrm{tp}(G) \leq P_G(K) \leq P_H(K) = \mathrm{tp}(H).
\]
(ii) In view of Corollary~\ref{cor:normal} we have 
\[
\mathrm{tp}(G/N) = \min_{N \leq H\leq G}\, P_{G/N}(H/N) = \min_{N \leq H\leq G} \,P_G(H) \geq \mathrm{tp}(G).
\]
(iii) Let $X = H/N$, where $N \unlhd H \leq G$. Then, by Parts~(i) and (ii), 
\[
\mathrm{tp}(G) \leq \mathrm{tp}(H) \leq\mathrm{tp}(H/N) = \mathrm{tp}(X),
\]
whence our result.

\noindent (iv) Suppose that $H\leq G$ is not normal in $G$, and that $(G:H) = n$. Then the number $m$ of trivial connected components of the graph $\Gamma^G_H$ satisfies $m = (N_G(H):H) < n$. Thus, there exists at least one non-trivial component $\Delta_\sigma \cong K_{t_\sigma, t_\sigma}$ in $\Gamma^G_H$, and we have $t_\sigma \geq p$, since $t_\sigma \big\vert \vert H\vert$ by Proposition~\ref{Prop:PFormula}. Hence, by \eqref{Eq:PFormula}, 
\[
\mathrm{tp}(G) \,\leq \, P_G(H) = \prod_{\sigma = 1}^s t_\sigma!/t_\sigma^{t_\sigma} \,\leq \, p!/p^p.
\]
The proof is complete.
\end{proofof}

\noindent We next give the common transversal probabilities of various groups, which will be needed later on. 

\begin{lemma}
\label{Lem:A4-A5}
For an odd prime $p$ and any integer $n \geq 1,$ let 
\begin{equation}
\label{Eq:CpC2^n}
C_p \rtimes C_{2^n} = 
\left\langle a, b \,\big\vert\, a^p = b^{2^n} = 1,\, a^b = a^{-1}\right\rangle, 
\end{equation}
and let $D_n$ be the dihedral group of order $2n$. Then
\vspace{-2.5mm} 
\begin{enumerate}
\item[(i)] $\mathrm{tp}(C_p \rtimes C_{2^n}) = \left(\frac{1}{2}\right)^{\frac{p-1}{2}},$
\vspace{1.5mm}
\item[(ii)] $\mathrm{tp}(D_n) = \begin{cases} 
(\frac{1}{2})^{\frac{n-1}{2}} \text{ if $n$ is odd}, \\[1mm]
(\frac{1}{2})^{\frac{n-2}{2}} \text{ if $n$ is even}, 
\end{cases}$
\vspace{1.5mm}
\item[(iii)] $\mathrm{tp}(Q_{16}) = \mathrm{tp}(C_4 \rtimes C_4) = 1/2$,
\vspace{1.5mm}
\item[(iv)] $\mathrm{tp} (A_4) = 2/9$,
\vspace{1.5mm}
\item[(v)] $\mathrm{tp}(A_5) = (1/2)^{14}$,
\vspace{1.5mm}
\item[(vi)] $\mathrm{tp}(\mathrm{PSL_3(2)}) = (1/2)^{40}$,
\vspace{1.5mm}
\item[(vii)] $\mathrm{tp}(C_2^3 \rtimes C_7) = (1/2)^{12}$,
\vspace{1.5mm}
\item[(viii)] $\mathrm{tp}(C_3^2 \rtimes C_4) = (1/2)^8$,
\vspace{1.5mm}
\item[(ix)] $\mathrm{tp}(\mathrm{SL}_2(3)) = (2/9)^2$.
\end{enumerate}
\end{lemma}

\begin{proofof}
(i) Let $G_p = C_p\rtimes C_{2^n}$ be as in \eqref{Eq:CpC2^n}. Then the only subgroups of $G_p$ that are not normal are the conjugates of $\langle b\rangle \cong C_{2^n}$. To see this observe that $a$ centralises $b^2$. This is because $a^b = a^{-1}$ implies that $aba = b$, and thus $a^kba^k = b$ for every $k \geq 1$. Hence 
\[
ab^2a^{-1} = (aba^{-1})^2 = (ba^{p-2})^2 = ba^{p-2}ba^{p-2} = b^2.
\]
Thus $\langle b^2\rangle$, which is the unique maximal
subgroup of $\langle b\rangle$, is normal in $G$.
Since $\langle b^2\rangle \cong C_{2^{n-1}}$ is cyclic, all of its subgroups (that is, all proper subgroups of $\langle b\rangle$) are also normal in $G_p$. Now let $H\leq G_p$
 be an arbitrary subgroup, and let $S$ be a $2$-Sylow subgroup of $H$. If $S \leq \langle b^2\rangle$ then, by the above, $S$ is normal in $G_p$, and we have $H = S$ or $H = S \langle a\rangle$, so $H\unlhd G_p$ in both cases. There remains the case that $S = \langle b\rangle^x$ for some $x\in G_p$. Then either $H = S = \langle b\rangle^x$ is a conjugate of $\langle b\rangle$, or we have $H = G_p$, which is again a normal subgroup of $G_p$. Our claim follows. Hence, by Corollary \ref{cor:iso} we have $\mathrm{tp}(G_p) = P_{G_p}(\langle b\rangle)$.

\noindent Now note that $N_{G_p}(\langle b\rangle) = \langle b\rangle$ while for any $x\in G_p$ we get $\langle b\rangle^x \cap \langle b\rangle = \langle b^2\rangle $. Hence $| \langle b\rangle :\langle b\rangle^x \cap \langle b\rangle| = 2 $ and the $t-$vector of $\langle b\rangle $ is $(\underbrace{2, 2, \cdots, 2}_{(p-1)/2}, 1)$.
 Hence $P_{G_p}(\langle b\rangle) = (\frac{1}{2})^{\frac{p-1}{2}}$ and the first part of the lemma follows.

\noindent (ii) We let 
\[
G = \left\langle a, b \,\big\vert\, a^2 = b^{n} = 1,\, b^a = b^{-1}\right\rangle \,\cong\, D_n, 
\]
and work by induction on $n$. The claim is true for $n = 3$ and $n = 4$ according to a GAP computation~\cite{GAP4}.
Suppose that the claim has been established for all
positive integers $j$ such that $3 \leq j < n$
and let $K$ be a non-normal subgroup of $G = D_n$.
If $|K| > 2$, then $K$ is dihedral and thus $K \cong D_k$
for some non-trivial proper divisor $k$ of $n$.
Moreover, let 
\[
N \coloneqq \mathrm{core}_G(K) = K \cap C, 
\]
where $C = \langle b\rangle$ is the normal cyclic subgroup of index $2$ in $D_n$
generated by a rotation by $\frac{2 \pi}{n}$ degrees. Since $G/N = \langle b^{n/k},\, a\rangle \cong D_{n/k}$, Corollary~\ref{cor:normal} plus the induction hypothesis yield
\[
P_G(K) = P_{G/N}(K/N) \geq \mathrm{tp}(D_{n/k}) = 
\left(\frac{1}{2}\right)^{\frac{\frac{n}{k}-i}{2}},\quad i = \left.\begin{cases} 1, & n/k \equiv 1\,(2)\\2,& n/k \equiv 0\,(2)\end{cases}\right\}.
\]
On the other hand, $\mathrm{Aut}(D_n)$ acts transitively
on the (non-normal) involutions of $D_n$. To see this,
recall that $D_n$ is generated by a rotation $r$ and a reflection $s$
subject to the relation $r^s = r^{-1}$.
Now consider the map $\phi_i : D_n \to D_n$ fixing every rotation 
and mapping $s r^j$ to $s r^{i+j}$, where the exponents work modulo $n$.
Then $\phi_i$ is a homomorphism (the reader is invited to check this detail) and since it is injective,
it is an automorphism of $D_n$.
Finally, notice that if $sr^k$, $sr^{\ell}$ are reflections,
then $\phi_{\ell} \circ \phi^{-1}_{k}$ sends $sr^k$ to $sr^{\ell}$.

Thus $P_G(H)$ has the same value for all non-normal
subgroups $H$ of $G$ of order $2$ by Corollary~\ref{cor:iso}. Let $H$ be one such subgroup.
Then
\[
P_G(H) = \begin{cases} 
(\frac{1}{2})^{\frac{n-1}{2}} \text{ if $n$ is odd}, \\[1.5mm]
(\frac{1}{2})^{\frac{n-2}{2}} \text{ if $n$ is even} 
\end{cases}
\]
by Corollary~\ref{Cor:Hprime}, where the distinction between cases is explained by the fact that $H$ is self-normalising if $n$ is odd (and thus $m = 1$ in that case),
while $m = (N_G(H) : H) = 2$ if $n$ is even.
Since $\mathrm{tp}(D_{\frac{n}{k}}) > P_G(H)$ in either case, $\mathrm{tp}(G)$ is attained for a non-normal cyclic subgroup of order $2$, and the induction is complete.

\noindent For Parts~(iii)--(ix) of Lemma~\ref{Lem:A4-A5} 
we use a GAP-routine to compute $\mathrm{tp}(G)$.
\footnote{The web address\, \url{https://fourier.math.uoc.gr/~marial/tp-Code}\, 
contains our GAP-code for computing $\mathrm{tp}(-)$.}
\end{proofof}

\noindent Our next result connects the function $\mathrm{tp}(-)$ to solubility of the corresponding finite group.

\begin{proposition}\label{Prop:solubilitycriterion}
Let $G$ be a finite group such that $\mathrm{tp}(G) > (1/2)^{40}$. Then $G$ is soluble or has a section isomorphic to $A_5$. Moreover, this result is best possible. 
\end{proposition}

\begin{proofof}
We argue by induction on the order of $G$.
Our claim is obviously true if $G$ is the trivial group, so that our induction begins. Assume now that $\vert G\vert \geq 2$, that $\mathrm{tp}(G) > (1/2)^{40}$, that $G$ is non-abelian, and that our claim holds for groups of smaller order. Suppose first that $G$ is a non-abelian simple group, and let $H$ be a subgroup of order $2$ in $G$ 
(existence of $H$ is guaranteed by the celebrated Feit-Thompson Odd Order Theorem~\cite{FeitThomp}). Then, by Corollary~\ref{Cor:Hprime} and our hypothesis,
\[
P_G(H) = 
\left(\frac{1}{2}\right)^{\frac{n-m}{2}} \,\geq \,
\mathrm{tp}(G) \,>\, \left(\frac{1}{2}\right)^{40},
\]
where $n = (G : H)$ and $m = (N_G(H) : H)$.
Thus $\frac{n-m}{2} < 40$, or equivalently $n - m < 80$.
It follows that 
\begin{equation}\label{Eq:56}
|G| - |N_G(H)| < 160.
\end{equation}
Since $G$ is non-abelian simple, 
no proper subgroup can have index less than $5$,
so $|N_G(H)| \leq \frac{|G|}{5}$ and thus
\begin{equation}\label{Eq:45}
|G| - |N_G(H)| \geq \frac{4|G|}{5}.
\end{equation}
Combining~\eqref{Eq:45} and~\eqref{Eq:56} yields
$\frac{4|G|}{5} < 160$, or equivalently $|G| < 200$.
However, the only non-abelian simple group of order less than $200$ are $A_5$ (of order $60$) and $\mathrm{PSL}_3(2)$ (of order $168$). Now $A_5$ satisfies $\mathrm{tp}(A_5) = (1/2)^{14} > (1/2)^{40}$ by Part~(v) of Lemma~\ref{Lem:A4-A5}, but trivially has a section isomorphic to $A_5$, thus is acceptable, while $\mathrm{tp}(\mathrm{PSL}_3(2)) = (1/2)^{40}$, so that $G \cong \mathrm{PSL}_3(2)$ is ruled out by our hypothesis. Thus, $G \cong A_5$ for $G$ non-abelian simple, and the claim follows in that case.

\noindent Consequently, we may assume further that $G$ is not simple, thus $G$
has a proper non-trivial normal subgroup $N$.
By Parts~(i) and (ii) of Proposition~\ref{Prop:structproperties}, we have $\mathrm{tp}(N) \geq\mathrm{tp}(G) > 2^{-40}$ and $\mathrm{tp}(G/N) \geq\mathrm{tp}(G) > 2^{-40}$,
thus, by the induction hypothesis applied to $N$ and $G/N$, 
each of $N$ and $G/N$ is either soluble or has a section isomorphic to $A_5$. If both are soluble, so is $G$. If $N$ has a section isomorphic to $A_5$, then so does $G$ while, if $G/N$ has a section isomorphic to $A_5$, this section lifts to give a corresponding section isomorphic to $A_5$ in $G$. This 
completes the induction. Finally, since $G = \mathrm{PSL}_3(2)$ is not soluble (in fact, $G$ is non-abelian simple), $5 \nmid \vert G\vert$ (so that $G$ cannot contain a section isomorphic to $A_5$), and $\mathrm{tp}(G) = (1/2)^{40}$ by Part~(vi) of Lemma~\ref{Lem:A4-A5}, our result is indeed best possible. 
\end{proofof}

\noindent We now turn our attention to the connection of $\mathrm{tp}(-)$ and supersolubility. The proof of Theorem~\ref{Thm:supersol} below is considerably harder than that of Proposition~\ref{Prop:solubilitycriterion}, which is mainly due to the fact that the class of supersoluble groups is not closed under extensions 
(otherwise every metabelian group would be supersoluble,
but $A_4$ is a counterexample).

\begin{theorem}\label{Thm:supersol}
Let $G$ be a finite group with $\mathrm{tp}(G) > (1/2)^8$. 
Then either $G$ is supersoluble, or $G$ has a section isomorphic to $A_4$. Moreover, the bound $(1/2)^8$ is sharp.
\end{theorem}

\begin{proofof}
We induce on the order of $G$. If $G$ is trivial, then the theorem clearly holds, so that the induction starts. Suppose that $\vert G\vert \geq 2$, that $G$ is non-abelian, that $\mathrm{tp}(G) > 2^{-8}$, and that the result holds for groups of smaller order. 
Since $\mathrm{tp}(G) > (1/2)^8 > (1/2)^{40}$, our group $G$ is soluble by Proposition~\ref{Prop:solubilitycriterion}. 
Let $A$ be a minimal normal subgroup of $G$. 
Then $A$ is an elementary abelian group of order $p^r$ 
for some prime $p$ and some positive integer $r$ 
(see, for instance, Satz~9.13 in \cite[Chap.~I]{Huppert1}). 
Since, by Part~(ii) of Proposition~\ref{Prop:structproperties} plus our hypothesis concerning $G$, 
\[
(1/2)^8< \mathrm{tp}(G) \leq \mathrm{tp}(G/A),
\]
the inductive hypothesis applies to the quotient $G/A$. 
Hence, either $G/A$ is supersoluble or $G/A$ has a section isomorphic to $A_4$. 
In the latter case, the section in question lifts isomorphically to a section in $G$; 
thus, we may suppose that $G/A$ is supersoluble.

\noindent If $A$ is cyclic, then $G$ is also supersoluble (see, for instance,~\cite[7.2.14]{Scott}). 
Thus, we may suppose further that $A$ is not cyclic; that is, $r > 1$ and $A \nleq Z(G)$. 
If $A \leq \Phi(G)$ then $G/A$ supersoluble would imply that $G$ itself is supersoluble 
owing to the fact that supersoluble groups comprise a saturated formation 
(see, for instance, Satz~8.6(a) in \cite[Chap.~VI]{Huppert1}). 
Consequently, we may also assume that $A \nleq \Phi(G)$, 
so that there exists a maximal subgroup $M$ of $G$ with $A \nleq M$; 
in particular, $A M = G$. As $A$ is normal in $G$, we have 
\[
(A\cap M)^g = A^g \cap M^g = A \cap M^g = A \cap M,\quad (g\in G),
\]
where $M = M^g$, comes from the fact that $M M^g \neq G$ by a theorem of Ore's (see Satz~3.9 in \cite[Chap.~II]{Huppert1}). 
Thus, $A \cap M \unlhd G$, and since $A\cap M < A$, minimality of $A$ forces 
$A \cap M = 1$, so that $G$ splits over $A$. 
Furthermore, $G/A \cong M$ and thus $M$ is supersoluble. 
Let $T$ be a minimal normal subgroup of $M$. 
Then $T$ is a cyclic group of order $q$ for some prime $q$, 
while $N_G(T) \geq M$. 
Since $M$ is maximal in $G$, 
either $N_G(T) = M$ or $T \unlhd G$. 
In the second case, invoking Proposition~\ref{Prop:structproperties}(ii), 
our assumption that $\mathrm{tp}(G) > (1/2)^8$, and the induction hypothesis, we deduce 
that either $G/T$ (and thus $G$) is supersoluble, or $G/T$, and thus $G$, has a section isomorphic to $A_4$. 
We may therefore assume that $N_G(T) = M$. 
If $(M : T) = t$, then Corollary~\ref{Cor:Hprime} implies that 
\begin{equation} \label{Eq:ssin}
\left(\frac{1}{2}\right)^8 < \mathrm{tp}(G) \leq P_G(T) = \left(\frac{q!}{q^q}\right)^{\frac{t(p^r-1)}{q}}, 
\end{equation}
where we note that
\[
(G:M) = (A \rtimes M : M) = \vert A\vert = p^r, 
\]
and that the exponent $t(p^r-1)/q$ is integral (see Corollary~\ref{Cor:Hprime}). 
Suppose first that $q \geq 5$. 
Then $q!/q^q \leq 5!/5^5 = 24/625$ and thus
\[
\left(\frac{q!}{q^q}\right)^2 < 
\left(\frac{1}{2}\right)^8 < \mathrm{tp}(G)
 \leq P_G(T) = \left(\frac{q!}{q^q}\right)^{\frac{t(p^r-1)}{q}}. 
\]
The above inequality forces the exponent of the
right-hand-side to be equal to $1$ so that
$q = t (p^r-1)$. As $r >1$ we conclude that 
$t = 1$ and so $p = 2$ and $q$ is a Mersenne prime
(thus $r$ is prime as well).
But the only Mersenne prime $q$ such that
$q! / q^q > (1/2)^8$ is $q = 7$, which implies that $|G| = 56$. 
Using GAP we see that the only non-supersoluble group of order $56$ 
is the Frobenius group $C_2^3 \rtimes C_7$, 
whose transversal probability is $1/4096 < 2^{-8}$ 
by Part~(vii) of Lemma~\ref{Lem:A4-A5}.

\noindent Now suppose that $q = 3$. Then 
\[
\left(\frac{3!}{3^3}\right)^4 < 
\left(\frac{1}{2}\right)^8 < \mathrm{tp}(G)
 \leq P_G(T) = \left(\frac{3!}{3^3}\right)^{\frac{t(p^r-1)}{3}}. 
\]
So $\frac{t(p^r-1)}{3} \leq 3$. 
As $ r >1$ and $(p^r -1) \geq 3$, we conclude that one of the following occurs:
\[
\begin{cases}
t = 1 \text{ and } p^r = 4, 8, 9,\\[1mm]
t = 2 \text{ and } p^r = 4,\\[1mm]
t = 3 \text{ and } p^r = 4.
\end{cases}
\]
So in all cases $t \cdot p^r \leq 12$ and thus $|G| = 3 \cdot t \cdot p^r \leq 36 $. Using GAP, we see that the non-supersoluble groups of order at most $36$ satisfying $\mathrm{tp} > (1/2)^8$ are $A_4, \mathrm{SL}_2(3), S_4, C_2 \times A_4, C_3.A_4$, and $C_3 \times A_4$, all of which have a section isomorphic to $A_4$ (actually, all but $S_4$ have a quotient isomorphic to $A_4$).

\noindent We should mention here that among the non-supersoluble groups of order $36$ is the group $G = C_3^2 \rtimes C_4$ whose transversal probability is exactly $(1/2)^8$ and $G$ contains no section isomorphic to $A_4$; see Part~(viii) of Lemma~\ref{Lem:A4-A5}. This implies that our bound, when proved, will be sharp.

\noindent We are left with the case where $q = 2$. Clearly, for \eqref{Eq:ssin} to hold, we must have $ \frac{t(p^r-1)}{2} \leq 7$. We still have $ r > 1$ and $ p^r -1 \geq 3$, thus we get one of the following cases:
\[
\begin{cases}
t = 1 \text{ and } p^r = 4, 8, 9,\\[1mm]
t = 2 \text{ and } p^r = 4,8,\\[1mm]
t = 3 \text{ and } p^r = 4,\\[1mm]
t = 4 \text{ and } p^r = 4.
\end{cases}
\]
In all cases, we have $t \cdot p^r \leq 16 $, thus $|G| = 2 \cdot t \cdot p^r \leq 32 $. As we have already checked, 
there is no counterexample to our claim among the groups of order at most $36$, so the proof is complete.
\end{proofof}

\noindent The last theorem of this subsection relates $\mathrm{tp}(-)$ and nilpotency. 
In order to express ourselves concisely and
avoid repetition, we adopt a local convention.
We will say that the group $G$ has a \emph{bad section}
if it has a section isomorphic to one of the groups
in the following set: $\{A_4, D_3, D_5, D_7 \}$.

\begin{theorem}\label{Thm:nilp}
Let $G$ be a finite group with 
$\mathrm{tp}(G) > 4/81 = (2/9)^2$. 
Then either $G$ is nilpotent 
or $G$ has a bad section. 
Furthermore, the bound is sharp.
\end{theorem}

\begin{proofof}
We induce on the order of $G$. 
If $G$ is trivial then the theorem clearly holds, 
thus the induction begins. 
Suppose that $\vert G\vert \geq 2$, 
that $\mathrm{tp}(G) > 4/81$, 
and that the result holds for groups of smaller order. 
Since $\mathrm{tp}(G) > 4/81 > (1/2)^8$, 
Theorem~\ref{Thm:supersol} implies that either $G$ is supersoluble 
or has a section isomorphic to $A_4$. 
Thus we may assume that our group is supersoluble.

\noindent We now argue that $\Phi(G) = 1$
and that $G$ has a unique minimal normal
subgroup, say $A$, which must moreover
be cyclic of prime order since $G$ is supersoluble.
\begin{itemize}
\item To justify the first assertion, suppose
instead that $\Phi(G) > 1$. Then
\[
4/81< \mathrm{tp}(G) \leq \mathrm{tp}(G/\Phi(G)),
\]
and thus the inductive hypothesis applied
to $G/\Phi(G)$ gives that $G/\Phi(G)$ is nilpotent
or $G/\Phi(G)$ has a bad section. If $G/\Phi(G)$ is nilpotent,
then $G$ is nilpotent and we are done. On the other
hand, if $G/\Phi(G)$ has a bad section, then
the section in question lifts isomorphically to a section in $G$
and thus $G$ itself has a bad section
and we are done in that case too.

\item To justify the second assertion,
suppose that there exist minimal normal subgroups
$A \neq B$. Then $G/A$ may be assumed to
be nilpotent and the same holds for $G/B$.
But nilpotent groups comprise a formation
thus $G/(A\cap B)$ is nilpotent, hence
$G$ is nilpotent (since $A \cap B = 1$) completing the proof.
\end{itemize}

\noindent So let $A$ be the unique minimal normal
subgroup of $G$. Since $\Phi(G) = 1$,
there exists a maximal subgroup $M$ not
containing $A$ and thus $A \cap M = 1$.
It follows that $M$ complements $A$ in $G$,
hence $G = A \rtimes M$.
Now observe that $M$ is necessarily core-free
owing to the fact that $A$ is the unique minimal normal
subgroup of $G$.
Then $C_G(A)$ is normal in $G$ and contains $A$.
By Dedekind's lemma we have $C_G(A) = AK$,
where $K \coloneqq C_M(A)$. Thus $K$ is normal
in $M$ and is centralised by $A$, 
hence $K \unlhd G$. 
It follows that $K$ is trivial and $A$
is self-centralising.
By the $N/C$-theorem $M \cong G/A$
embeds isomorphically as a subgroup of 
the automorphism group of $A$ which is cyclic
of order $p-1$.
Moreover, every non-trivial element of $A$
is a generator of $A$, thus for all $a \in A$
with $a \neq 1$ we have 
\[
C_G(a) = C_G(\langle a \rangle) = C_G(A) = A \leq A.
\]
Thus (see Theorem 6.4 in \cite{Isa}) $G$ is isomorphic to a Frobenius group
with $A$ the Frobenius kernel 
and $M$ a Frobenius complement 
being cyclic of order $d$, where $d \mid p-1$.
By Corollary~\ref{Cor:Frob} we have
\[
P_G(M) = \left(\frac{d!}{d^d}\right)^{\frac{p-1}{d}} \geq 
\mathrm{tp}(G) > 
\frac{4}{81}.
\]
If $d = p-1$, then the fact that 
$d!/d^d > 4/81$ forces $d \leq 4$
and thus $d \in \{1,2,4\}$.
Then $G$ is isomorphic to $C_2$ or to $D_3$ or to the Frobenius group of order $20$ respectively, 
so $G$ is nilpotent or $G$ has a bad section 
(in the last case because it has a subgroup isomorphic to $D_5$).
We may therefore assume that $d < p - 1$,
hence $(p-1)/d \geq 2$. 
Thus 
\[
\left(\frac{d!}{d^d}\right)^2 \geq
\left(\frac{d!}{d^d}\right)^
{\frac{p-1}{d}} >\,\,
\left(\frac{2}{9}\right)^2
\]
forcing $d!/d^d > 2/9$.
That is only possible if $d = 2$
whence $G$ is itself a dihedral group.
But the inequality 
\[
\left(\frac{1}{2}\right)^{\frac{p-1}{2}} >
\left(\frac{2}{9}\right)^2
\]
only holds for $p \leq 7$ in which case $G$ has a bad section.
The proof is complete.

\noindent The semi-direct product $C_7 \rtimes C_3$ 
has common transversal probability equal to $4/81$ 
and is neither nilpotent nor does it contain a bad section. 
Thus, our result is indeed sharp.
\end{proofof}

\subsection{Structural characterisation of certain $\mathrm{tp}$-values}

\noindent 
We saw in the previous section, how the knowledge  of the  range of values of the invariant  $\mathrm{tp}(G)$ provides structural information for the group $G$. But what can be said about specific values of $\mathrm{tp}(G)$? 
For example,  we have seen that if $G$ is not a Dedekind group then there exists a non-normal subgroup $H$ and thus $\mathrm{tp}(G) \leq P_G(H) \leq 1/2$, so it is natural to wonder if we can characterise all groups with $\mathrm{tp}(G)= 1/2$. 
We will actually show that if $\mathrm{tp}(G)= p!/p^p$ for a prime $p$ then $p=2$ and the group $G$ is one of a  very specific list of groups (cf. Corollary~\ref{Cor:tp = 1/2} below). We will also show, that like the value  $p!/p^p$ for $p$ odd prime that can't be attained, also $\mathrm{tp}$ can never equal $ \frac{p!}{p^p} \cdot \frac{q!}{q^q}$ for distinct primes $p <q$ (see Theorem~\ref{Thm:p<q} below). On the other hand if $p=q=2$ we fully characterise groups with transversal probability $\mathrm{tp}(G)=1/4$ (cf. Theorem~\ref{thm:tp = 1/4}).

Our first result is purely arithmetic and is of independent interest.

\begin{theorem}
 \label{Thm:prodp_i}
Assume that 
\begin{equation}\label{eq:prodp_i}
 \prod_{i = 1}^n\frac{t_i!}{t_i^{t_i}} = \prod_{i = 1}^k\frac{p_i!}{p_i^{p_i}}, 
\end{equation}
where the numbers $t_i >1$ are positive integers and the $p_i$ are distinct primes. 
Then $k = n$ and, after appropriate rearrangement, $t_i = p_i$ for all $i = 1, \ldots, n$.
\end{theorem}

\begin{proofof}
We induce on $k$.
Assume $k = 1$, that is 
$\prod_{i = 1}^n\frac{t_i!}{t_i^{t_i}} = \frac{p!}{p^{p}}$.
Then there exists a $t_{\sigma}$ such that $p \mid t_{\sigma}$ and without loss we may assume $p \mid t_1$. 
Now, the sequence $a_n = \frac{n!}{n^n} $ is monotonically decreasing, as 
\[
\frac{a_{n+1}}{a_n} = 
\left(\frac{n}{n+1}\right)^n = 
\left(\frac{1}{1+1/n}\right)^n 
< 1,
\]
while clearly $a_n < 1$ for all $n > 1$. 
Hence 
\[
\frac{p!}{p^p} = 
\prod_{i = 1}^n\frac{t_i!}{t_i^{t_i}} \leq \frac{t_1!}{t_1^{t_1 }} \leq 
\frac{p!}{p^p}.
\]
Therefore $t_1 = p$ and $n = 1$.

\noindent Assume now that the inductive hypothesis holds for all values less than $k$ and we will prove it for $k$. 
Thus we have $\prod_{i = 1}^n\frac{t_i!}{t_i^{t_i}} = \prod_{i = 1}^k\frac{p_i!}{p_i^{p_i}}$ for distinct primes $\{p_i\}_{i = 1}^k$ and we assume that $\{t_i \}$ and $\{p_i\}$ are written in a decreasing order. So $p: = p_1$ is the largest prime involved in the right hand side and $t_1 \geq t_2 \geq \ldots \geq t_n$. 
Clearly $p >2$ and equation \eqref{eq:prodp_i} implies
\begin{equation} \label{eq:prodp_i-1}
 p^{p-1} \prod_{i = 2}^k p_i^{p_i} \prod_{i = 1}^n t_i! = (p-1)! \prod_{i = 2}^k p_i! \prod_{i = 1}^n t_i^{t_i}.
 \end{equation}
Hence $p^{p-1}$ divides $\prod_{i = 1}^n t_i^{t_i}$, and thus $p$ divides at least one of the $t_i$'s. 
If $p = t_{\sigma}$ for some $\sigma = 1, \ldots , n$, then we are done by the inductive hypothesis. To see this, divide both sides of \eqref{eq:prodp_i} by $p!/p^{p}$ to get
\[
\prod_{i = 1, i \neq \sigma}^n\frac{t_i!}{t_i^{t_i}} = \prod_{i = 2}^k\frac{p_i!}{p_i^{p_i}},
\]
where the inductive hypothesis now applies.

\noindent So we may assume that $t_i \neq p $ for all $i = 1, \ldots , n$, while the fact that $p$ divides some $t_i$ implies that $t_1 \geq 2p$.

\noindent For any $x = m/n \in \mathbb{Q}$
let $v_p(x)$ denote the $p$-adic valuation of $x = m/n$, that is, if $x \in \mathbb{Z}$ then $v_p(x)$ is the exponent of the largest power of $p$ that divides $x$ and $v_p(x) = 0$ if $p \nmid x$, 
while for $x = m/n$ we  write  $v_p(x) = v_p(m/n) = v_p(m) - v_p(n)$.
Then observe that for any integer $t$ we have that 
\[
\mbox{$v_p(t! /t^t) < 0 $ \, \, if and only if \, \, $p \mid t$.} 
\]
Now let $q$ be the biggest prime smaller than $t_1$. Then $q > p$ because $t_1 \geq 2p$ and according to Bertrand's postulate there exists a prime in the interval $(p, 2p)$; cf., for instance, \cite[Thm.~418]{HW}. 
Hence 
\begin{equation}
p < q \leq t_1.
\end{equation}
Assume now that $t_1=q$ and observe that $v_q(t_i!/t_i^{t_i})= 1-q $ for all $i $ with $t_i=q$ (there is at least one such, namely $t_1$), while for the remaining we have  $v_q(t_i!/t_i^{t_i})=0$, since $t_1$ is the biggest of the $t_i$'s.  Hence equation \eqref{eq:prodp_i} implies 
\[
\sum_{t_i=q} (1-q) = \sum_{i=1}^n v_q(t_i!/t_i^{t_i}) = \sum_{i=1}^k v_q(\frac{p_i!}{p_i^{p_i}}) = 0
\]
which is clearly absurd.  Hence $q <t_1$.
In addition, applying Bertrand's postulate again for the prime $q$, we conclude that $t_1 < 2q$ (or else a bigger prime than $q$ exists in $(q, 2q)$ and this prime would be smaller than $t_1$
contradicting the choice of $q$). 
Hence 
\begin{equation}
p < q <  t_1 < 2q .
\end{equation}
Therefore $q \nmid t_1$. 
So $v_q(t_1!/t_1^{t_1}) = 1$, while 
$v_q(\prod_{i = 1}^k\frac{p_i!}{p_i^{p_i}}) = 0$ as $q > p$ and $p$ is the largest of the $p_i$'s.

\noindent Thus, in view of \eqref{eq:prodp_i}, we get 
\[
1+ \sum_{i = 2}^n v_q(t_i!/t_i^{t_i}) = v_q(\prod_{i = 1}^k\frac{p_i!}{p_i^{p_i}}) = 0.
\]
We conclude that 
$\sum_{i = 2}^n v_q(t_i!/t_i^{t_i})<0$. 
Hence there exists $t_s$, for some $s = 2, \ldots, n$ so that $q \mid t_s$. But for every $i = 1, \ldots, n $ we have $t_i \leq t_1 < 2q$. Hence there exists some $s = 2, \ldots, n$ with $t_s = q$ and thus $v_q(t_s!/t_s^{t_s}) = 1-q$. 
We conclude that 
\[
2-q+ \sum_{i = 2, i \neq s}^n v_q(t_i!/t_i^{t_i}) = 0.
\]
Hence there exist at least $q-2$ among the $t_i$'s that are greater than $q$. These along with $t_1$ and $t_s$ provide at least $q$ elements among the $t_i$'s that are $\geq q$. Hence $t_i \geq q$ for all $i = 1, \ldots, q$ (as these are $\geq$ of the $q$-previously picked $t_i$'s) and therefore
\[
\sum_{i = 1}^qt_i \geq q^2 > p^2 > \frac{p(p+1)}{2} 
\]
while 
\[
\sum_{i = 1}^{k} p_i \leq \sum_{k = 1}^{p_1} k = \frac{p(p+1)}{2} 
\]
We conclude that the vector $(t_1, t_2, \ldots, t_q, \ldots, t_n)$
weakly majorises the vector $(p_1, p_2, \ldots, p_k)$, while all the hypothesis of Proposition \ref{Prop:strict} (see appendix) are satisfied. Hence 
\[
\prod_{i = 1}^n\frac{t_i!}{t_i^{t_i}}< \prod_{i = 1}^k\frac{p_i!}{p_i^{p_i}} 
\]
contradicting the hypothesis of the theorem. 
This final contradiction implies that there does exist $t_i$ with $t_i = p$ and so, as we have seen, the inductive hypothesis completes the proof of the theorem. 
\end{proofof}

\noindent Combining Theorem~\ref{Thm:prodp_i} with Proposition~\ref{Prop:PFormula} we get the following.

\begin{corollary}
\label{Cor:p<q}
Assume that $P_G(H) = \frac{p!}{p^p} \cdot \frac{q!}{q^q}$ with distinct primes $p < q$, where $G$ is a finite group with subgroup $H$. Then the $t$-vector of $H$ in $G$ is given by $(q, p, 1, \cdots, 1);$
in particular, $pq\, \big\vert\, |H|$. 
\end{corollary}

\noindent With the help of Corollary~\ref{Cor:p<q}, we can now prove a result excluding certain rational numbers from the range of the function $\mathrm{tp}(-)$.

\begin{theorem}\label{Thm:p<q}
There exists no finite group $G$
such that $\mathrm{tp}(G) = \frac{p!}{p^p} \cdot \frac{q!}{q^q}$ for primes $p < q$.
\end{theorem}

\begin{proofof}
Let $G$ be a counterexample of smallest possible order,
and let $H$ be a proper subgroup of $G$ for which $\mathrm{tp}(G)$ is realised; that is, 
\[
P_G(H) = \mathrm{tp}(G) = \frac{p!}{p^p} \cdot \frac{q!}{q^q}
\]
for distinct primes $p < q$. We argue that $H$ is core-free. Suppose not, and let $N \coloneqq \mathrm{core}_G(H)$. 
Then $N > 1$ and $P_{G/N}(H/N) = P_G(H)$ by Corollary~\ref{cor:normal}.
Moreover, by Part~(ii) of Proposition~\ref{Prop:structproperties}, we have 
\[
P_{G/N}(H/N) \geq \mathrm{tp}(G/N) \geq \mathrm{tp}(G) = P_G(H).
\]
Therefore, $\mathrm{tp}(G/N) = \mathrm{tp}(G)$, contradicting the minimality of $G$. Thus $H$ is core-free, as claimed.

\noindent The prime $q$ divides $|H|$ according to Corollary \ref{Cor:p<q}. Let $K < H$ be of order $q$. Then $K$ is not normal in $G$, as $H$ is core-free. Also 
\[
P_G(K) = \left(\frac{q!}{q^q}\right)^{\frac{n-m}{q}},
\]
in view of Corollary \ref{Cor:Hprime}, where $n = (G:K)$, $m = (N_G(K):K)$, $m\mid n$, and $m \neq n$.
In addition, we have 
\[
\left(\frac{q!}{q^q}\right)^{\frac{n-m}{q}} = P_G(K) \geq \mathrm{tp}(G) = \frac{p!}{p^p} \cdot \frac{q!}{q^q}.
\]
But $p <q$, and thus $ \frac{p!}{p^p} > \frac{q!}{q^q}$. So the above inequality holds 
if, and only if, $\frac{n-m}{q} = 1$, which in turn, as $m \mid n$, implies that $m \mid q$. 
Hence either $m = 1$ and $n = q+1$, or $m = q$ and $n = 2q$. 

\noindent In the first case, we get $K = N_{G}(K)$, and thus $K = N_H(K)$, while $|G| = q(q+1)$. 
Furthermore, by Proposition~\ref{Prop:PGSub} and Corollary~\ref{Cor:Hprime}, we have 
\[
\frac{q!}{q^q} = P_G(K) \leq P_H(K) = \left( \frac{q!}{q^q}\right)^{\frac{(H:K)-1}{q}}.
\]
Hence, $0 \leq \frac{(H:K) - 1}{q} \leq 1$. As $H \neq K$, 
we necessarily have $(H:K) - 1 = q$, 
and thus $(H:K) = q+1 = (G:K)$ or, equivalently, $G = H$, 
contradicting the fact that we had chosen $H$ as a proper subgroup of $G$. 
So the first case, where $m = 1$, does not occur.

\noindent Consequently, we must have $m = q$, $n = 2q$, and $|G| = 2q^2$. 
This forces $p = 2$, the only other prime involved in the order of $G$. 
As $H < G$, while $2$ and $q$ both divide $H$ by Corollary~\ref{Cor:p<q}, 
we get $|H| = 2q$, thus $K$ is a normal subgroup of $H$. 
So $H \leq N_G(K)$, and thus 
\[
2 = (H:K)\, \big\vert\, (N_G(K):K) = q. 
\]
This final contradiction now implies the theorem.
\end{proofof}

\noindent We can now show that if for a specific subgroup $H \leq G$ we know that $P_G(H) = \frac{p!}{p^p}$ for some prime $p,$  then $H$ has a very restricted place inside $G$.

\begin{lemma}\label{Lem:p/p}
Let $G$ be a finite group, and let $H\leq G$ be a subgroup. 
If $P_G(H) = \frac{p!}{p^p}$ for some prime $p,$ 
then $p \,\big\vert\, |H|$ and one of the following occurs:
\vspace{-2mm}
\begin{enumerate}
\item[(i)] $H = N_G(H)$ and $(G : H) = p + 1,$ or 
\vspace{1.5mm}
\item[(ii)] $(N_G(H):H) = p,$ while $(G : H) = 2p$.
\end{enumerate}
In particular, $P_G(H) = 1/2$ implies that $H$ is a subgroup of even order, whose index in $G$ is either $3$ or $4$.
\end{lemma}

\begin{proofof}
Suppose that $P_G(H) = p!/p^p = \prod_\sigma t_\sigma!/t_\sigma^{t_\sigma}$ 
for some prime $p$, and let $m = (N_G(H) : H)$. 
Combining Proposition~\ref{Prop:PFormula} with Theorem~\ref{Thm:prodp_i}, 
we see that $\Gamma^G_H$ has precisely one non-trivial component 
$\Delta_{\sigma_1} \cong K_{t_{\sigma_1},\, t_{\sigma_1}}$, and that $t_{\sigma_1} = p$. 
Hence, $p\, \big\vert\, \vert H\vert$, and we have $n = (G:H) = m+p$. 
Since $m \mid n$, we get $m \mid p$. 
Thus, either $m = 1$ (that is, $N_G(H) = H$) and $n = p+1$, or $m = p$ and $n = 2p$, 
giving $(N_G(H) : H) = p$ in the second case.
\end{proofof}

\noindent Our next result takes a major step towards classifying groups $G$ 
where $\mathrm{tp}(G) = p!/p^p$ for some prime $p$.

\begin{theorem}\label{thm:classify-p}
Let $G$ be a finite group. 
Assume that, for every non-normal subgroup 
$H \leq G,$ we have $P_G(H) = p!/p^p$ 
with some fixed prime $p$. 
Then one of the following holds: 
\vspace{-2mm}
\begin{enumerate}
\item[(i)] $G$ is a Dedekind group. 
In this case, all subgroups are normal and $\mathrm{tp}(G) = 1$. 
\vspace{1.5mm}
\item[(ii)] $G \cong C_3 \rtimes C_{2^n}$, for some integer $n\geq 1$. 
\vspace{1.5mm}
\item[(iii)] $G \cong D_4$.
\vspace{1.5mm}
\item[(iv)] $G \cong Q_{16}$.
\vspace{1.5mm}
\item[(v)] $G \cong C_4 \rtimes C_4$.
\end{enumerate}
In Cases~{\em (ii)--(v),} we have $p = 2$ and $\mathrm{tp}(G) = 1/2$.
\end{theorem}

\begin{proofof}
Clearly, $\mathrm{tp}(G) = 1$ 
if, and only if, every subgroup of $G$ is normal 
or, equivalently, if, and only if, $G$ is a Dedekind group.
Thus, discarding Case~(i), we assume from now on that $\mathrm{tp}(G) < 1$. 

\noindent Fix some non-normal subgroup $H$ of $G$. Then, by hypothesis, $P_G(H) = p!/p^p$; thus, by Lemma~\ref{Lem:p/p}, we have $(G:H) = p+1$ or $(G:H) = 2p$, as well as $p\, \big\vert\, \vert H\vert$. 

\noindent \textbf{Claim:}\label{claim}
Every proper subgroup of $H$ is normal in $G$ 
and $H \cong C_{p^k}$, a cyclic group of order $p^k$ for some positive integer $k$.

\begin{proofofclaim} If $(G:H) = p+1$ and $K < H$, 
then clearly $(G:K) \neq p+1$. Also, $(G:K) \neq 2p$, or else $2p = r(p+1)$ for some integer $r\geq 2$. 
This implies that $p \mid r$, while $2 = r+r/p > 2$, a contradiction. If $(G:H) = 2p$ and $K < H$, 
then $(G:K) \geq 4p$; in particular, $(G:K) \neq p+1, 2p$. 
Hence, in both cases, $K$ is normal, 
and the first part of the claim follows.

\noindent If $M_1, M_2$ are distinct maximal subgroups of $H$, 
then they are normal in $G$, and their product is $M_1M_2 = H$. 
This would imply that $H$ is normal in $G$, contrary to our assumption. 
Hence, $H$ has a unique maximal subgroup; 
thus, $H$ is cyclic of prime power order. 
Moreover, as $p \,\big\vert\, |H|$, 
we conclude that $H \cong C_{p^k}$ for some $k\geq 1$, and our claim follows.
\end{proofofclaim}

\noindent For the rest of the proof, we shall need to distinguish two cases.

\noindent {\bf Case 1:} \emph{$p$ is odd.} As we have seen, we either have $(G:H) = p+1$ and $H = N_G(H)$, 
or $(G:H) = 2p$ and $\vert N_G(H)| = p^{k+1}$. 
Hence, either $|G| = p^k(p+1)$, or $|G| = 2p^{k+1}$.
Therefore, in either case, $2 \big\vert |G|$, so that there exists some $x \in G$ of order 2. Let $T = \langle x\rangle$. Clearly, $T$ 
is not a subgroup of $N_G(H)$ (in either case). 
This implies that $T$ cannot be normal in $G$. 
Indeed, if $T \unlhd G$,
then $x \in Z(G)$, and so $T$ would be a subgroup of $N_G(H)$. It follows that $p\, \big\vert\, \vert T\vert = 2$, forcing $p = 2$, which contradicts our case assumption. Consequently, $p$ has to be even.

\noindent {\bf Case 2:} $p = 2$. Here, there are two subcases:\\[1.5mm] 
(a) $N_G(H) = H$ and $(G:H) = 3$, or\, (b) $(N_G(H):H) = 2 = (G:N_G(H))$. 

\noindent \textit{Subcase~ $(a)$.} By our claim proved above, $H \cong C_{2^n}$ for some positive integer $n$. Since $(G:H) = 3$, so that $\vert G\vert = 3\cdot 2^n$, it follows that, if $Q$ is a Sylow $3$-subgroup of $G$, then $Q \cong C_3$. Moreover, $Q$ is normal in $G$, since otherwise $P_G(Q) = 1/2$, so that $Q$ would have to have even order by Lemma~\ref{Lem:p/p}. By a well-known theorem of Zassenhaus (see, for instance, \cite[Hauptsatz~I.18.1]{Huppert1}), we have 
\[
G\,\cong\, \left\langle a, b\,\big\vert\, a^3 = b^{2^n} = 1,\, a^b = a^{-1}\right\rangle \,\cong C_3 \rtimes C_{2^n},\quad n\geq 1,
\] 
since $\mathrm{tp}(G) < 1$, whence Case~(ii) of the theorem. 

\noindent \textit{Subcase~$(b)$.} Now suppose that
\[
(G : N_G(H)) = 2 = (N_G(H) : H),
\]
where $H \cong C_{2^n}$ for some $n\geq 1$, so that $G$ is a $2$-group. At this point, we are tasked with determining the collection 
of all $2$-groups $G$ of order $\vert G\vert \geq 8$, which enjoy the following property:
\[
\mbox{\it All subgroups of $G$ are normal in $G$ except
some subgroups of index $4$.}\tag*{($\dagger$)}
\]

Note that we are not allowing for any ambiguity in ($\dagger$);
there \emph{must} exist non-normal subgroups of index $4$.
The only group of order $8$ with this property 
is clearly $D_4$ and we use GAP to find the groups
of order $16$: they are $Q_{16}$ and \texttt{SmallGroup(16,4)},
which is of type $C_4 \rtimes C_4$.

\noindent We now argue that those are the only $2$-groups with this property.
To prove the assertion, we shall show that, if $G$
satisfies $(\dagger)$, then $|G| \leq 16$. 
Indeed, suppose that $G$ is a minimal counterexample to our last claim, 
and let $H$ be a non-normal subgroup of index $4$ in $G$. 
Let $N$ be a subgroup of order $2$ in $H$, 
and observe that $N$ is normal in $G$, 
since $\vert G\vert \geq 32$, and $G$ satisfies $(\dagger)$. 
By the Correspondence Theorem, $G/N$ satisfies $(\dagger)$ 
and is not a counterexample. 
Thus, $\vert G/N\vert \leq 16$, and so $\vert G\vert \leq 32$ 
which implies that $\vert G\vert = 32$. 
Using GAP again, we check that no group of order $32$ satisfies $(\dagger)$, 
and we have reached the desired contradiction. 
Hence, our claim holds, and the only $2$ groups satisfying $(\dagger)$ are those exhibited above.
\end{proofof}

\noindent As a consequence of Theorem~\ref{thm:classify-p}, we can now determine, up to isomorphism, all groups $G$ with $\mathrm{tp}(G) = 1/2$.

\begin{corollary}\label{Cor:tp = 1/2}
Let $G$ be a finite group. 
We have $\mathrm{tp}(G) = 1/2$ if, and only if, 
one of the following occurs:
\vspace{-2mm}
\begin{enumerate}
\item[(i)] $G = C_3 \rtimes C_{2^n}$, for some integer $n\geq 1,$
\vspace{1.5mm}
\item[(ii)] $G = D_4,$
\vspace{1.5mm}
\item[(iii)] $G = Q_{16},$
\vspace{1.5mm}
\item[(iv)] $G = C_4 \rtimes C_4$.
\end{enumerate}
\end{corollary}
\begin{proofof}
This follows from Theorem~\ref{thm:classify-p} in conjunction with Corollary~\ref{Cor:Prod}.
\end{proofof}

\noindent As a further application of Theorem~\ref{thm:classify-p}, we have the following.

\begin{corollary}
\label{Cor:NonAbOdd}
If $G$ is a non-abelian group of odd order, then $\mathrm{tp}(G) \leq 4/81$.
\end{corollary} 

\begin{proofof}
Let $H \leq G$ be a non-normal subgroup. As 
\[
m = (N_G(H): H) < (G:H) = n, 
\]
the graph $\Gamma^G_H$ has at least one non-trivial component $\Delta_\sigma \cong K_{t_\sigma,\,t_\sigma}$, and since $1 < t_\sigma \big\vert \vert H\vert \big\vert \vert G\vert$ by Proposition~\ref{Prop:PFormula}, and $G$ has odd order, we have $t_\sigma \geq 3$. It follows that, for any such group $G$, either $\mathrm{tp}(G) \leq 4/81$, or $P_G(H) = 2/9$ for all subgroups $H$ of $G$ which are not normal in $G$. In the second case, Theorem~\ref{thm:classify-p} applies, so that $G$ would have to be one of the groups listed in (i)--(v). However, by the well-known classification of Dedekind groups, Case~(i) does not apply, while Cases (ii)--(v) are ruled out as $G$ has odd order by hypothesis. Hence, $\mathrm{tp}(G) \leq 4/81$, as claimed.
\end{proofof}

\noindent Our next result determines all groups $G$ with $\mathrm{tp}(G) = 1/4$.

\begin{theorem}\label{thm:tp = 1/4}
Let $G$ be a finite group. 
We have $\mathrm{tp}(G) = 1/4$ if, and only if, 
one of the following occurs:
\begin{enumerate}
\item[(i)] $G = \left\langle a, b \,\vert\, a^5 = b^{2^k} = 1,\, a^b = a^{-1}\right\rangle$ for some integer $k \geq 1$.
\vspace{1.5mm}
\item[(ii)] $G$ has a normal cyclic $2$-subgroup $M = C_{2^k}$ 
for some integer $k \geq 0$, 
so that $G/M$ is one of the groups 
$D_6$, $M_4(2)$, $C_4 \circ D_4$, $C_2 \times D_4$, 
or $C^2_2 \rtimes C_4$. 
\end{enumerate}
\end{theorem}

\begin{proofof}
The proof is similar to that of Theorem~\ref{thm:classify-p} so we only sketch some steps.
Assume first that $P_G(H) = 1/4$ for some subgroup $H$ of $G$, where $1/4 = \prod_{\sigma = 1}^s\frac{(t_\sigma)!}{t_\sigma^{t_\sigma}}$. 
Then note that we necessarily have two values $t_{\sigma_1} = t_{\sigma_2} = 2$, 
while the rest are all $1$. 
The reason is that if there exist $t_{\sigma} \geq 3$, 
then $\frac{(t_\sigma)!}{t_\sigma^{t_{\sigma}}} \leq 2/9$ and thus 
$\prod_{\sigma = 1}^s\frac{(t_\sigma)!}{t_\sigma^{t_\sigma}} \leq 2/9 $. 
Hence, $2 \big\vert |H|$ and $4 + m = n$, where $m, n$ are, as usual, 
the index of $H$ in $N_G(H)$ and in $G$, respectively. 
Furthermore, we have $m \mid n$, and thus $m \mid 4$. 
Hence $m = 1, 2, $ or $4$. 
In conclusion, if $P_G(H) = 1/4$, 
then $H$ is group of even order whose indices $m$, $n$ equal one of the following:
\[
m = 1 \text{ and } n = 5, \, \, m = 2 \text{ and } n = 6, \, \text{ or }\, m = 4 \text{ and } n = 8.
\]

\noindent Assume now that $\mathrm{tp}(G) = 1/4$. Then for every non-normal subgroup $T$ of $G$ 
we should have $P_G(T) = 1/4 $ or $P_G(T) = 1/2$. 
Hence, the only possible values for the index $(G : T)$ are $3, 4, 5, 6,$ and $8$. 
(The first two values occur in the case that $P_G(T) = 1/2$, according to Lemma~\ref{Lem:p/p}.)
As $\mathrm{tp}(G) = 1/4$, there exists a non-normal subgroup $H \leq G$ with $P_G(H) = 1/4$, 
and thus $(G : H) \in \{ 5, 6, 8\}$, 
while $H$ has even order. 
Observe that if $M$ is any proper subgroup of $H$, 
then $(G : M) \notin \{3, 4, 5, 6, 8\}$. 
Therefore any proper subgroup of $H$ is normal in $G$. 
We conclude that $H = C_{2^k}$ for some integer $k \geq 1$. 
We now distinguish three cases.

\noindent {\bf Case 1:} \emph{$(G : H) = 5$ and $H = N_G(H)$.}

\noindent If $C$ is a $5$-Sylow subgroup of $G$, 
then $P_G(C) = 1$, as $|C| = 5$ is odd. 
Hence $C$ is a normal subgroup of $G$. 
If $M$ is the maximal subgroup of $H$, 
then $C \cdot M = C \times M$. 
So $M $ is a central subgroup of $G$. 
Let $C = \langle a \,\vert\, a^5 = 1\rangle$ and 
$H = \langle b \,\vert\, b^{2^k} = 1\rangle$. 
Then $b$ acts as an automorphism of order $2$ on $C$. 
We conclude that 
 \[
G = \left\langle a, b\, \big\vert\, a^5 = b^{2^k} = 1,\, a^b = a^{-1}\right\rangle 
\]
According to Lemma \ref{Lem:A4-A5}(i), the above group has transversal probability equal to 1/4 and Case 1 is completed.

\noindent {\bf Case 2:} \emph{$(G : H) = 6$ and $(N_G(H) : H) = 2$.}

\noindent In this case the $3$-Sylow subgroup $C$ of $G$ is a normal subgroup of $G$, or else 
we would have $P_G(C) $ being a power of $2/9$.
If $|H| = 2$ then $G = D_6$. 
So we may assume that $H = C_{2^k}$ with $k > 1$.
Let $M = \langle b^2\rangle$ be the maximal subgroup of $H$. 
Then $1 \neq M \unlhd G$.
In addition $1/4 = P_G(H) = P_{G/M}(H/M)$, 
while $\mathrm{tp}(G) \leq \mathrm{tp}(G/M)$. 
We conclude that $G/M$ is a group of order $12$ 
whose common transversal probability is $1/4$. 
Using GAP, we find that $G/M = D_6$. 

\noindent {\bf Case 3:} \emph{$(G : H) = 8$ and $(N_G(H) : H) = 4$.}

\noindent In this case $G$ is a $2$-group. 
If $|H| = 2$ then $G$ is a group of order $16$. 
Using GAP again, we see that the only groups of order $16$ 
that have transversal probability $1/4$ 
are $M_4(2), \, C_4 \circ D_4, \, C_2 \times D_4$, and $ C^2_2 \rtimes C_4$. 
We may assume $|H| > 2$, and we write $M$ for its unique maximal subgroup. 
Then $M$ is normal in $G$. 
As earlier, we get $\mathrm{tp}(G/M) = 1/4$, while $|G/M| = 16$. 
Hence $G/M$ is one of the previously mentioned groups. 

\noindent This completes the proof in Case 3, and the theorem follows. 
\end{proofof}

\noindent Proposition~\ref{Prop:solubilitycriterion} might have left the reader (as it did us) with the nagging feeling that a refinement might perhaps be possible, which connects $\mathrm{tp}(-)$ with the derived length of a soluble group. The final result of this section is a first indication that such a result might indeed exist, although we have been unable so far to find it. The reader should also note that the bound obtained in the next result is the same one achieved in Theorem~\ref{Thm:nilp}.

\begin{proposition}\label{Prop:derived}
If $G$ is a non-abelian group with $\mathrm{tp}(G) > 4/81,$ 
then $G$ has derived length $2$. Furthermore, the bound is sharp.
\end{proposition}

\begin{proofof}
We induce on the order of $G$.
Note that $G$ cannot have odd order according to Corollary~\ref{Cor:NonAbOdd}. 
Furthermore, if $H$ is a subgroup of $G$ of odd order, 
then $H$ must be abelian; 
otherwise, combining Corollary~\ref{Cor:NonAbOdd} with Part~(i) of Proposition~\ref{Prop:structproperties}, 
we would get 
\[
4/81< \mathrm{tp}(G) \leq \mathrm{tp}(H) \leq 4/81, 
\]
a contradiction. Also, since 
\[
\mathrm{tp}(G) \,>\, 4/81 \,>\, (1/2)^{14} = \mathrm{tp}(A_5) 
\]
$G$ is necessarily soluble by Proposition~\ref{Prop:solubilitycriterion} and Proposition~\ref{Prop:structproperties}(iii).

\noindent
Next, we argue that $G$ has a unique minimal
normal subgroup. Assume otherwise and let $N_1$, $N_2$
be distinct minimal normal subgroups of $G$. 
Since $G/N_i$ is either abelian, or of derived length $2$ by the induction hypothesis plus Proposition~\ref{Prop:structproperties}(ii), 
we see that $\left(G/N_1\right)'$ is abelian, 
as is $\left(G/N_2\right)'$.
We have 
\[
(G/N_i)' = G' N_i/N_i\, \cong \,G'/ G' \cap N_i,\quad (i = 1, 2)
\]
and, since $G' \cap N_i \unlhd G$, either $G' \cap N_i = 1$, or $N_i \leq G'$, due to minimality of $N_i$. If $G' \cap N_1 = 1$, then $G'$ is abelian, and the induction is complete. Arguing similarly for $N_2$, we deduce that $N_1, N_2 \leq G'$. Therefore both $G'/N_1$ and $G'/N_2$ are abelian groups.
Now notice that $N_1 \cap N_2 = 1$, again due to minimality of the $N_i$. 
Moreover, the group $G'\big/(N_1 \cap N_2)$,
which is isomorphic to a subgroup of $G'/N_1 \times G'/N_2$,
is abelian as well. 
It follows that $G'$ is abelian, thus $G$ has 
derived length $2$, as required.

\noindent We may therefore assume that $G$ has a unique minimal normal subgroup. 
It follows that for every subgroup $H$ of order $2$ in $G$, 
either $H$ is normal and thus central in $G$
(and moreover there can exist only one such subgroup by the previous observation) 
or, by Part~(i) of Proposition~\ref{Prop:structproperties} plus Corollary~\ref{Cor:Hprime}, 
we have 
\[
2^{\frac{n-m}{2}}< 81/4,
\]
where $n = |G|/2$ and $m = |N_G(H)|/2$. 
This implies $n - m \leq 8$ and thus $|G| - |N_G(H)| \leq 16$.
Since $|N_G(H)| \leq |G|/2$, 
the previous inequality implies that $|G| \leq 32$.
A GAP check confirms that every group of order $\leq 32$
with $\mathrm{tp}(G) > 4/81$
has derived length $\leq 2$ 
and thus the proof is complete in that case.

\noindent It therefore suffices to treat the case where
a Sylow $2$-subgroup of $G$ has only one 
minimal subgroup which is moreover the unique minimal
normal subgroup of $G$.
If $G$ is a $2$-group, then it is either cyclic
or a generalised quaternion group (these being the only $2$-groups with a unique minimal subgroup); the first case is ruled out by our hypothesis that $G$ is non-abelian. In the second case, the derived length equals $2$. We may thus assume that $G$ is not a $2$-group. 

\noindent Let $P$ be a non-trivial Sylow $p$-subgroup of $G$ for some odd prime $p$ and let $T \leq P $ be of order $p$. As the only minimal normal 
subgroup of $G$ is a $2$-group, the cyclic group 
$T$ is not normal in $G$, thus $N_G(T) < G$ and, by Proposition~\ref{Prop:structproperties}(i) and Corollary~\ref{Cor:Hprime}, 
\[
4/81 < 
\mathrm{tp}(G) \leq 
\mathrm{tp}(T) \leq 
P_G(T) = 
\left(\frac{p!}{p^p}\right)^{\frac{n-m}{p}},
\]
where $n = |G|/p$ and $m = |N_G(T)|/p < n$. 
Since $\frac{n-m}{p} \geq 1$ and $p!/p^p \leq 24/625 < 4/81$ for $p\geq 5$, 
we must have $p = 3$ and $n-m = 3$. 
Moreover, since $m\mid n$, we conclude that $m\mid 3$. 
Thus, either $m = 1$ and $\vert G\vert = 12$, or $m = 3$ and $\vert G\vert = 18$. 
As we have already checked, no group of order $12$ or $18$ has derived length $> 2$ 
and transversal probability $> 4/81$, and our proof is complete. 

\noindent Finally, we note that the transversal probability 
of $\mathrm{SL}_2(3)$ equals $4/81$ by Part~(ix) of Lemma~\ref{Lem:A4-A5}, 
while its derived length is $3$; 
so that our bound is indeed sharp as claimed.
\end{proofof}

\subsection{Group extensions and the function $\mathrm{tp}(-)$}
\label{Sec:tpExt}

\noindent
The aim in this section is to see how the function $\mathrm{tp}(-)$ behaves with respect to group extensions.

\noindent Let $G$ and $K$ be groups, and let $(V, W)$ be a factor system for $G$ by $K$. In detail, this means that $V: K\rightarrow \mathrm{Aut}(G)$ and $W: K\times K\rightarrow G$ are maps, such that 
\begin{align}
V(k_2) \circ V(k_1) \,& = \, i_{W(k_1, k_2)}\circ V(k_1 k_2),\quad (k_1, k_2 \in K) \label{Eq:FactorProp1}\\[1mm]
W(k_1, k_2k_3) W(k_2, k_3) \,& = \, W(k_1k_2, k_3) V(k_3)(W(k_1, k_2)),\quad (k_1, k_2, k_3 \in K), \label{Eq:FactorProp2}
\end{align}
where $i_g(x) = g^{-1} x g$ for $x, g\in G$, so that $i_g$ is the inner automorphism of $G$ associated with the element $g\in G$. Let $\widehat{G}$ be the extension of $G$ by $K$ associated with $(V, W)$; that is, $\widehat{G} = K\times G$ as a set, with group law given by
\begin{equation}
\label{Eq:GLExt}
(k_1, g_1)\cdot (k_2, g_2) = (k_1 k_2,\,W(k_1, k_2) V(k_2)(g_1)\, g_2),\quad (k_1, k_2\in K;\, g_1, g_2\in G),
\end{equation}
and we have a short exact sequence 
\[
\begin{CD}
1 @>>> G @>{\iota}>> \widehat{G} @>{\pi}>> K @>>> 1,
\end{CD}
\]
where $\iota(g) = (1,\, W(1, 1)^{-1} g)$ and $\pi(k, g) = k)$, and the associated map $\varphi_E: K\rightarrow \widehat{G}$ is given by $\varphi_E(k) = (k, 1)$; cf. \cite[Sec.~9.4]{Scott}, in particular Statement~9.4.5. We shall need the following properties of a factor system.
\begin{lemma}
\label{Lem:FactorSet}
Let $(V, W)$ be a factor system for $G$ by $K$. Then
\vspace{-2mm}
\begin{enumerate}
\item[(i)] $V(1) = i_{W(1, 1)},$
\vspace{1.5mm}
\item[(ii)] $W(k, 1) = W(1, 1),\quad (k\in K),$
\vspace{1.5mm}
\item[(iii)] $W(1, k) = V(k)(W(1, 1)),\quad (k\in K)$.
\end{enumerate}
\end{lemma}

\begin{proofof}
(i) Setting $k_1 = k_2 = 1$ in \eqref{Eq:FactorProp1} gives 
\[
V(1) \circ V(1) = i_{W(1, 1)} \circ V(1).
\]
Thus, 
\[
V(1) = V(1) \circ V(1) \circ V(1)^{-1} = i_{W(1, 1)} \circ V(1) \circ V(1)^{-1} = i_{W(1, 1)},
\]
as desired.

\noindent (ii) Setting $k_1 = k$ and $k_2 = k_3 = 1$ in \eqref{Eq:FactorProp2}, we get
\[
W(k, 1) W(1, 1) = W(k, 1) V(1)(W(k, 1)).
\]
Applying Part~(i), this equation may be rewritten as 
\[
W(k, 1) W(1, 1) = W(k, 1) W(1, 1)^{-1} W(k, 1) W(1, 1).
\]
Multiplying the last equation from the left by $W(k, 1)^{-1}$ and from the right by $W(1, 1)^{-1}$, the result follows.

\noindent (iii) Setting $k_1 = k_2 = 1$ and $k_3 = k$ in \eqref{Eq:FactorProp2} yields
\[
W(1, k)^2 = W(1, k) V(k)(W(1, 1)),
\]
whence (iii). 
\end{proofof}

\begin{theorem}\label{Thm:tpExtEst}
Let $\widehat{G}$ be an extension of the group $G$ by the group $K$ with associated factor system $(V, W)$. Then
\begin{equation}
\label{Eq:tpExtEst}
\mathrm{tp}(\widehat{G})\, \leq\, \min_{H\leq G} \prod_{k\in K} P_G(H, V(k)(H)).
\end{equation}
\end{theorem}

\begin{proofof}
Let $H\leq G$ be a subgroup (with $G$ considered as a subgroup of $\widehat{G}$ via the embedding \nolinebreak $\iota$). Then, for $h\in H$ and $(k, g) \in \widehat{G}$, we have 
\begin{align*}
(k, g) \iota(h) & = (k, g) (1, W(1, 1)^{-1} h)\\[1mm]
& = (k,\, W(k, 1) V(1)(g) W(1, 1)^{-1} h)\\[1mm]
& = (k, \,W(k, 1) W(1, 1)^{-1} g W(1, 1) W(1, 1)^{-1} h)\\[1mm]
& = (k,\, W(k, 1) W(1, 1)^{-1} g h) = (k, g h), 
\end{align*}
where we have applied the definition of $\iota$ in the first step, \eqref{Eq:GLExt} in the second step, Lemma~\ref{Lem:FactorSet}(i) in Step~3, and Lemma~\ref{Lem:FactorSet}(ii) in the last step. Hence, 
\begin{equation}\label{Eq:LeftHCosetForm} 
(k, g) \cdot H = (k,\, g H),\quad ((k, g) \in \widehat{G},\, H\leq G).
\end{equation}
Similarly, making use of the definition of $\iota$, the group law \eqref{Eq:GLExt}, and Part~(iii) of Lemma~\ref{Lem:FactorSet}, we find that, for $h\in H$ and $(k, g) \in \widehat{G}$,
\begin{align*}
\iota(h) (k, g) & = (1, W(1, 1)^{-1} h)(k, g)\\[1mm]
& = (k,\, W(1, k) V(k)(W(1, 1)^{-1} h) g)\\[1mm]
& = (k,\, W(1, k) V(k)(W(1, 1))^{-1} V(k)(h) g)\\[1mm]
& = (k,\, V(k)(h) g),
\end{align*}
implying 
\begin{equation}\label{Eq:RightHCosetForm}
H \cdot (k, g) = (k, V(k)(H) g),\quad ((k, g) \in \widehat{G},\, H\leq G).
\end{equation}
Now let
\[
\widetilde{T} = \{(k_i, g_j)\}_{(i, j)}\,\in\, \mathrm{DT}_{\widehat{G}}(H)
\]
be any two-sided transversal for $H$ in $\widehat{G}$. 
From either \eqref{Eq:LeftHCosetForm} or \eqref{Eq:RightHCosetForm}, when applied to $H\leq G$ and $(k_i, g_j)\in \widetilde{T}$, it is clear that \emph{all} elements of $K$ must occur as first component of an element in $\widetilde{T}$, since otherwise
\[
(k, 1) \,\not\in\, \bigcup_{(i, j)} (k_i, g_j) H = \widehat{G}
\]
for some $k\in K$, a contradiction. Given $k\in K$, consider the set 
\[
T(k) \coloneqq \left\{g\in G:\, (k, g) \in \widetilde{T}\right\}\, \subseteq\, G.
\]
We claim that $T(k) \in \mathrm{DT}_G(H, V(k)(H))$. 

\noindent First, let $k\in K$ be given, and let $g\in G$ be arbitrary. Then
\[
(k, g)\,\in\, (k_i, g_j) H = (k_i, g_j H) 
\]
for some $(i, j)$. This implies $k_i = k$ and $g \in g_j H$. Thus, $g_j \in T(k)$ and $g \in \bigcup_{g_\ell \in T(k)} g_\ell H$, so that 
\[
G = \bigcup_{g_\ell \in T(k)} g_\ell H.
\]
Similarly, given $g\in G$, we have 
\[
(k, g) \,\in\, H (k_i, g_j) = (k_i, V(k_i)(H) g_j)
\]
for some $(i, j)$, thus $k_i = k$, $g\in V(k)(H) g_j$, and $g_j \in T(k)$, so that 
\[
G = \bigcup_{g_\ell \in T(k)} V(k)(H) g_\ell.
\]
Next, suppose that $g_j H \cap g_\ell H \neq \emptyset$, where $g_j, g_\ell \in T(k)$. Then 
\begin{equation*}
(k, g_j) H\,\cap\, (k, g_\ell) H = (k, g_j H) \,\cap\, (k, g_\ell H) 
 = (k, g_j H \cap g_\ell H) \,\neq\, \emptyset,
\end{equation*}
implying $(k, g_j) = (k, g_\ell)$ by our hypothesis on $\widetilde{D}$, thus $g_j = g_\ell$. Hence, $T(k)$ is a left transversal for $H$ in $G$. Similarly, suppose that 
\[
V(k)(H) g_j \,\cap\, V(k)(H) g_\ell \,\neq\, \emptyset
\]
 for some $g_j, g_\ell \in T(k)$. Then we have 
\begin{equation*}
H (k, g_j) \,\cap\, H (k, g_\ell) = (k, V(k)(H) g_j)\,\cap\, (k, V(k)(H) g_\ell) \,
 = \, (k,\, V(k)(H) g_j \cap V(k)(H) g_\ell)\,\neq\,\emptyset,
\end{equation*}
so that, again, $g_j = g_\ell$. Consequently, $T(k)$ is also a right transversal for $V(k)(H)$ in $G$, therefore $T(k) \in \mathrm{DT}_G(H, V(k)(H))$, as claimed. 

\noindent Mapping $k\in K$ to $T(k)$ for given $\widetilde{T}$ thus gives a choice function
\[
f(\widetilde{T}): K\,\longrightarrow\,\bigsqcup_{k\in K} \mathrm{DT}_G(H,\, V(k)(H)),
\]
and, subsequently, sending $\widetilde{T} \in \mathrm{DT}_{\widehat{G}}(H)$ to $f(\widetilde{T})$, defines a map
\[
\widetilde{\Phi}: \mathrm{DT}_{\widehat{G}}(H)\,\longrightarrow\, \mathrm{CF}(K, H, G), 
\]
where, unsurprisingly, $\mathrm{CF}(K, H, G)$ denotes the set of all choice functions
\[
f: K\,\longrightarrow\, \bigsqcup_{k\in K} \mathrm{DT}_G(H,\, V(k)(H));
\]
that is, functions $f$ as above, such that $f(k) \in \mathrm{DT}_G(H,\,V(k)(H))$ for each $k\in K$. For later use we observe that, obviously, 
\[
\vert\mathrm{CF}(K, H, G)\vert = \prod_{k\in K} \vert \mathrm{DT}_G(H,\, V(k)(H))\vert.
\]
Next, we note that, if $\widetilde{T}_1, \widetilde{T}_2 \in \mathrm{DT}_{\widehat{G}}(H)$ are such that $\widetilde{\Phi}(\widetilde{T}_1) = f = \widetilde{\Phi}(\widetilde{T}_2)$ then, by definition of $\widetilde{\Phi}$,
\[
\widetilde{T}_1 = \bigcup_{k\in K} (k, f(k)) = \widetilde{T}_2,
\]
so that $\widetilde{\Phi}$ is injective. We want to show that $\widetilde{\Phi}$ is surjective as well. Let $f\in \mathrm{CF}(K, H, G)$ be given, and set 
\[
\widetilde{T}_f\coloneqq\,\bigcup_{k\in K} (k, f(k))\,\subseteq\, \widehat{G}.
\]
We claim that $\widetilde{T}_f \in \mathrm{DT}_{\widehat{G}}(H)$. By \eqref{Eq:LeftHCosetForm}, we have 
\begin{equation*}
\bigcup_{\widetilde{t} \in \widetilde{T}_f} \widetilde{t} \,H 
 = \bigcup_{k\in K}\, \bigcup_{g\in f(k)} (k, g) H 
 = \bigcup_{k\in K}\, \bigcup_{g\in f(k)}\, (k, gH) 
 = \bigcup_{k\in K} \left(k, \,\bigcup_{g\in f(k)} g H\right)
 = \bigcup_{k\in K} (k, G)
 = \widehat{G}. 
\end{equation*}
Similarly, by \eqref{Eq:RightHCosetForm}, we have 
\begin{multline*}
\bigcup_{\widetilde{t} \in \widetilde{T}_f} H \,\widetilde{t} = \bigcup_{k\in K}\, \bigcup_{g\in f(k)} H (k, g) 
 = \bigcup_{k\in K}\, \bigcup_{g\in f(k)} (k,\, V(k)(H)\, g) \\[1mm]
 = \bigcup_{k\in K} \left(k,\, \bigcup_{g\in f(k)} V(k)(H)\,g\right)
 = \bigcup_{k\in K} (k, G)
 = \widehat{G}.
\end{multline*}
Moreover, for $\widetilde{t}_1 = (k_1, g_1), \widetilde{t}_2 = (k_2, g_2)\in\widetilde{T}_f$, where $k_1, k_2 \in K$, $g_1 \in f(k_1)$, and $g_2\in f(k_2)$, the hypothesis 
\[ 
\widetilde{t}_1 H \,\cap\, \widetilde{t}_2 H = (k_1, g_1) H \,\cap\, (k_2, g_2) H = (k_1, g_1 H) \,\cap\, (k_2, g_2 H) \,\neq\, \emptyset 
\]
first implies $k_1 = k_2 = : k$, thus $g_1, g_2 \in f(k)$, as well as $g_1 H \cap g_2 H \neq \emptyset$, which forces $g_1 = g_2$, since $f(k)$ forms a left transversal for $H$ in $G$. Hence, $\widetilde{t}_1 = \widetilde{t}_2$. Also, with $\widetilde{t}_1, \widetilde{t}_2, k_1, k_2, g_1, g_2$ as above, the assumption that 
\begin{align*}
H \widetilde{t}_1 \,\cap\, H \widetilde{t}_2 \,& = \, H (k_1, g_1) \,\cap\, H (k_2, g_2) \\[1.5mm] 
& = (k_1, V(k_1)(H) g_1) \,\cap\, (k_2, V(k_2)(H) g_2) \,\neq \,\emptyset
\end{align*}
 implies $k_1 = k_2 = : k$, thus $g_1, g_2 \in f(k)$, and $V(k)(H) g_1 \cap V(k)(H) g_2 \neq \emptyset$, forcing $g_1 = g_2$, as $f(k)$ also forms a right transversal for $V(k)(H)$ in $G$. Hence, again, $\widetilde{t}_1 = \widetilde{t}_2$, and it follows that $\widetilde{T}_f \in \mathrm{DT}_{\widehat{G}}(H)$. Since $\widetilde{\Phi}(\widetilde{T}_f) = f$ by construction, we conclude that $\widetilde{\Phi}$ is surjective, thus a bijection. Therefore, 
\[
\vert \mathrm{DT}_{\widehat{G}}(H)\vert = \vert CF(K, H, G)\vert,
\] 
and we conclude that 
\begin{align*}
 \mathrm{tp}(\widehat{G})\,\leq \, \min_{H\leq G} P_{\widehat{G}}(H) 
 & = \min_{H\leq G} \frac{\vert \mathrm{DT}_{\widehat{G}}(H)\vert}{\vert H\vert^{(\widehat{G}:H)}}\\[1mm]
 & = \min_{H\leq G} \frac{\vert \mathrm{DT}_{\widehat{G}}(H)\vert}{\vert H\vert^{(G:H) \vert K\vert}} \\[1mm]
& = \min_{H\leq G} \prod_{k\in K} \frac{\vert \mathrm{DT}_G(H, V(k)(H))\vert}{\vert H\vert^{(G:H)}} \\[1mm]
& = \min_{H\leq G} \prod_{k\in K} P_G(H, V(k)(H)), 
\end{align*}
whence \eqref{Eq:tpExtEst}.
\end{proofof}

\begin{corollary}
\label{Cor:tpSemiDirProdEst}
Let $\widehat{G} = G \rtimes K,$ viewed as an internal semidirect product. Then we have
\begin{equation}
\label{Eq:tpSemiDir}
\mathrm{tp}(\widehat{G}) \leq \min_{H\leq G} \prod_{k\in K} P_G(H, H^k).
\end{equation}
\end{corollary}

\begin{proofof}
Suppose that the factor system $(V, W)$ in Theorem~\ref{Thm:tpExtEst} splits, so that $(V, W)$ is equivalent to some factor system $(V^\ast, W^\ast)$ such that $V^\ast: K \rightarrow \mathrm{Aut}(G)$ is a homomorphism, $W^\ast(k_1, k_2) = 1$ for all $k_1, k_2\in K$, and $\varphi_E$ is a section to the projection $\pi$. Identifying $g\in G$ with $\iota(g) = (1, g)$ and $k\in K$ with $\varphi_E(k) = (k, 1)$, we have 
\[
V(k)(g) = g^k,\quad (k\in K,\, g\in G),
\]
so that $\mathrm{DT}_G(H, V(k)(H)) = \mathrm{DT}_G(H, H^k)$, and thus 
\[
P_G(H, V(k)(H)) = P_G(H, H^k).
\]
Our claim follows now from Theorem~\ref{Thm:tpExtEst}.
\end{proofof}

\begin{corollary}
\label{Cor:tpDirProdEst}
Let $\widehat{G} = G_1\times \cdots \times G_r,$ where the $G_\rho$ are finite groups, 
and set $m \coloneqq \vert \widehat{G}\vert$. 
Then we have 
\begin{equation}
\label{Eq:tpDirProdEst}
\mathrm{tp}(\widehat{G}) \leq \min_{1\leq \rho \leq r} \mathrm{tp}(G_\rho)^{m/\vert G_\rho\vert}.
\end{equation}
\end{corollary}

\begin{proofof}
An immediate induction on $r$, starting with the trivial case where $r = 1$, reduces us to the case where $r = 2$; that is, to 
\begin{equation}
\label{Eq:tpDirProd2}
\mathrm{tp}(G\times K) \leq \min\left\{\mathrm{tp}(G)^{\vert K\vert},\, \mathrm{tp}(K)^{\vert G\vert}\right\}. 
\end{equation}
Moreover, as $\mathrm{tp}$ is an isomorphism invariant by Corollary~\ref{cor:iso}, it suffices to show for \eqref{Eq:tpDirProdEst} that
\[
\mathrm{tp}(G\times K) \,\leq \, \mathrm{tp}(G)^{\vert K \vert} 
\]
However, assuming that $[G, K] = 1$, Corollary~\ref{Cor:tpSemiDirProdEst} gives
\[
\mathrm{tp}(G\times K) \,\leq \, \min_{H\leq G} \prod_{k\in K} P_G(H) = \prod_{k\in K} \min_{H\leq G} P_G(H) = \mathrm{tp}(G)^{\vert K\vert},
\]
whence the result.
\end{proofof}

\noindent It is easy to construct examples where inequality~\eqref{Eq:tpDirProdEst} is sharp. 
For instance, let $G = S_3\times C_p$, where $p\geq 5$ is a prime number. 
Then the minimum on the right-hand side of~\eqref{Eq:tpDirProd2} equals $\mathrm{tp}(S_3)^p = 2^{-p}$. 
Since $p\nmid \vert S_3\vert$, all subgroups of $G$ are of the form $H = U\times C$, 
where $U\leq S_3$ and $C\leq C_p$. 
Discarding normal subgroups, and making use of Corollary~\ref{cor:iso}, 
we see that, as regards $\mathrm{tp}(G)$, we only need to check the subgroups 
$H_1 = \langle (1, 2)\rangle \times1$ and $H_2 = \langle (1, 2)\rangle \times C_p$, 
whose transversal probabilities are given by $P_G(H_1) = 2^{-p}$ and $P_G(H_2) = 1/2$, respectively. 
Hence, 
\[
\mathrm{tp}(G) = \min\left\{1,\, 1/2, \,2^{-p}\right\} = 2^{-p},
\] 
as desired. 

\section{Some problems and questions}
\label{Sec:Problems}
\noindent In this final section we outline several open problems, conjectures, 
and questions, hoping thereby to stimulate further research on this topic.

\begin{problem}
{\em Assume a block diagonal doubly stochastic matrix 
\[
T = \mathrm{diag}(T_1, \ldots, T_s) 
\]
is given where, for $i = 1, \ldots, s$, the matrix $T_i$ has order $t_i$, 
and its entries all equal $1/t_i$. 
Do there exist a finite group $G$ and subgroups $H, K$ of $G$ (of the same order), 
such that $\mathrm{per}(T) = P_G(H,K)$?
If, in addition, we know that the number of $t$'s equal to 1 is non-zero, 
do there exist a finite group $G$ and a subgroup $H$, 
such that $\mathrm{per}(T) = P_G(H)$?}
\end{problem}

\begin{problem}
{\em Suppose that $H \leq G$ and $N \unlhd G$, where $G$ is finite. 
Is it then always true that $P_G(H) \leq P_G(HN)$ 
(GAP computations confirm this up to $\vert G\vert = 200$)? 
Such a result would have an interesting consequence: 
combining it with Theorem~\ref{Thm:PHom}, it would follow that, 
for any group homomorphism $f: G \rightarrow K$, 
and with $N \coloneqq \mathrm{ker}(f)$,
\begin{equation}
\label{Eq:fGenIneq}
P_G(H) \,\leq \,P_G(HN) = P_{f(G)}(f(HN)) = P_{f(G)}(f(H)), 
\end{equation}
so that we would have a general inequality relating to the function $P(-)$ 
associated with each homomorphism $f$.
Moreover, it should be possible to characterise equality in \eqref{Eq:fGenIneq}.}
\end{problem}

\begin{problem}
{\em Characterise equality in \eqref{Eq:tpDirProd2}. 
In particular, equality should hold, if $(\vert G\vert, \vert K\vert) = 1$ 
(this seems plausible, and is supported by massive computational evidence).}
\end{problem}

\begin{problem}
{\em What is the appropriate setting for the quantity
$P_G(H)$ to make sense when $G$ is an infinite group
and $H$ is a finite index subgroup of $G$?}
\end{problem}

\noindent Recall that $\mathrm{cp}(G)$ is the commuting probability
of $G$, i.e. the probability that two randomly chosen
elements of $G$ commute. 
Regarding the relation between 
$\mathrm{tp}(G)$ and $\mathrm{cp}(G)$
we propose the following conjecture.

\begin{problem}
{\em Let $G$ be a finite group. 
Then $\mathrm{tp}(G) \leq \mathrm{cp}(G)$ 
except if $G$ is Dedekind non-abelian, 
in which case $\mathrm{cp}(G) / \mathrm{tp}(G) = 5/8$, 
or $G = Q_{16}$ in which case $\mathrm{cp}(G) / \mathrm{tp}(G) = 7/8$. Moreover,
if $\mathrm{cp}(G) = \mathrm{tp}(G)$, 
then either $G$ is abelian or 
\[
G \cong \left\langle a, b : a^3 = b^{2^n} = 1, a^b = a^{-1} \right\rangle
\]
for some positive integer $n$ and thus $\mathrm{tp}(G) = \mathrm{cp}(G) = 1/2$.}
\end{problem}

\begin{problem}
{\em Does Theorem~\ref{Thm:prodp_i} generalise? In particular, is the following true:
Assume that 
\[
\prod_{i = 1}^n f(t_i) = \prod_{j = 1}^k f(s_j),
\]
where $t_i, s_j$ are (distinct) integers greater than 1, for all appropriate
$i, j$ and $f(x) = \Gamma(x+1)/x^x$.
Can we then conclude that $n = k$ and $t_i = s_i$ for all $i$ after rearranging appropriately?}
\end{problem}

\begin{problem}
{\em Let $G$ be a finite group and suppose that
$K \leq H \leq G$. In Proposition~\ref{Prop:PGSub}
we saw that $P_G(K) \leq P_H(K)$.
Are $P_G(K)$ and $P_G(H)$ connected in some way?
If so, how?}
\end{problem}

\noindent Finally, as regards the quantity $\mathrm{tp}(G)$,
we ask the following which, in our opinion, is the
most important relevant question.

\begin{problem}
{\em Let $G$ be a finite group. 
Does there exist a prime divisor $p$ of $|G|$
and a cyclic $p$-subgroup $H$ such that 
$\mathrm{tp}(G) = P_G(H)$ 
and moreover $H$ has the property $(H : N) \leq p$, 
where $N = \mathrm{core}_G(H)$?}
\end{problem}

\noindent We think the answer is \enquote{yes}.

\appendix 
\section{Majorisation}
We begin by recalling basic definitions and concepts from
the theory of majorisation and refer the reader to the canonical
work on this topic~\cite{MOA} for further information. 
Towards the end of this section we will prove a result that we
have appealed to in the proof of Theorem~\ref{Thm:prodp_i}.

\noindent Fix a positive integer $s$ and let $\mathbb{R}_{+} \coloneqq [0,+\infty)$.
For any $x = \left(x_{1}, \ldots, x_{s}\right) \in \mathbb{R}^{s}$, let
\[
x_{[1]} \geq \ldots \geq x_{[s]}
\]
denote the components of $x$ in decreasing order, and let
\[
x_{\downarrow} = \left(x_{[1]}, \ldots, x_{[s]}\right)
\]
denote the decreasing rearrangement of $x$.

\begin{definition}
For $x, y \in \mathbb{R}^s$ we write $x \prec y$ and say
that $x$ is majorised by $y$ (or that $y$ majorises $x$) if
\[
\sum_{i = 1}^{k} x_{[i]} \leq \sum_{i = 1}^{k} y_{[i]}, \quad k = 1, \ldots, s-1, 
\quad\quad \text{and} 
\quad\quad \sum_{i = 1}^{s} x_{[i]} = \sum_{i = 1}^{s} y_{[i]} .
\]
\end{definition}

\noindent Inequality in the final equality in the definition above leads to
the concept of weak majorisation.

\begin{definition}
For $x, y \in \mathbb{R}^n$ we write $x \prec_{\mathrm{w}} y$ and say
that $x$ is weakly majorised by $y$ (or that $y$ weakly majorises $x$) if
\[
\sum_{i = 1}^{k} x_{[i]} \leq \sum_{i = 1}^{k} y_{[i]}
\]
for all $k \in [s]$.
\end{definition}

\noindent We present now the fundamental concept of Schur-convexity/concavity.

\begin{definition}
A real-valued function $\phi$ defined on a set $\mathcal{A} \subset \mathbb{R}^{n}$ 
is said to be Schur-convex on $\mathcal{A}$ if
$x \prec y$ on $\mathcal{A}$ implies that $\phi(x) \leq \phi(y)$.
If, in addition, $\phi(x) < \phi(y)$ whenever $x \prec y$ but $x$ is not a permutation of $y$, 
then $\phi$ is said to be strictly Schur-convex on $\mathcal{A}$. 
Similarly, $\phi$ is said to be Schur-concave on $\mathcal{A}$ if
$x \prec y$ on $\mathcal{A}$ implies that $\phi(x) \geq \phi(y)$
and $\phi$ is strictly Schur-concave on $\mathcal{A}$ if strict inequality $\phi(x) > \phi(y)$ holds 
when $x$ is not a permutation of $y$.
\end{definition}

\noindent Schur's fundamental result asserts that if $\phi$ is a convex real function,
then the function $\sum_{i = 1}^s \phi(x_i)$ is Schur-convex.
Since we are interested in products, the following result helps
make the transition from sums to products.

\begin{theorem}[{\cite[Prop.~E1, pp. 105--106]{MOA}}]\label{Thm:PropE1}
Let $f$ be a continuous non-negative function defined on an interval $I \subset \mathbb{R}$. 
Then
\[
h(x) = \prod_{i = 1}^{s} f\left(x_{i}\right), \quad x \in I^{s}
\]
is Schur-convex/concave on $I^{s}$ if, and only if, $\log f$ is convex/concave on $I$. 
Moreover, $h$ is strictly Schur-convex/concave on $I^{s}$ 
if, and only if, $\log f$ is strictly convex/concave on $I$.
\end{theorem}

\noindent Of course, it follows from the above result that 
$\phi(x) = \prod_{i = 1}^{n} g\left(x_{i}\right)$ is (strictly) Schur-concave 
if, and only if, $\log g$ is (strictly) concave.

\noindent Finally, we mention the following crucial result.

\begin{theorem}[{\cite[A.9.a, p. 177]{MOA}}]\label{Thm:A9a}
Let $x, y \in \mathbb{R}^s$.
If $x \prec_{\mathrm{w}} y,$ then there exists a vector $v$ such that
$x \prec v$ and $v \leq y;$ that is, we have $v_i \leq y_i$ for all $i \in [s]$.
\end{theorem}

\noindent For $x\geq 0$, recall that 
\[
f(x) \coloneqq \frac{\Gamma(x+1)}{x^x},
\]
where $\Gamma(x) = \int_{0}^{\infty} t^{x-1} e^{-t} d t$
is the $\Gamma$--function.
Note that $\log f$ is strictly concave on $I = [0,+\infty)$ 
owing to the fact that $(\log f)'' < 0$ which we have established in section~\ref{Sec:PImproved Bound}.

\noindent It is a cconsequence of Theorem~\ref{Thm:PropE1} and the comment immediately following it
that, since $\log f(x)$ is strictly concave on $I = [0,+\infty)$,
the function $h : I^s \rightarrow \mathbb{R}$ with
\begin{equation}\label{Eq:h}
h(x_1, \ldots, x_s) = \prod_{\sigma = 1}^s f(x_{\sigma})
\end{equation}
is strictly Schur-concave on $I^s$.

\begin{proposition}\label{Prop:strict}
Let $x, y \in I^s$ and suppose that $x = x_{\downarrow},$ $y = y_{\downarrow}$.
Assume that $x \prec_{\mathrm{w}} y,$
that is, $x_1 + \ldots x_i \leq y_1 + \ldots y_i$ for all $i \in [s]$.
Suppose that the coordinates of $x$ are pairwise distinct.
Moreover, assume that there exists a $k \in [s]$
such that $x_1 + \ldots x_j < y_1 + \ldots y_j$
for all $j \in \{k, \ldots, s\},$ 
so that strict inequality holds in the weak
majorisation order between $x$ and $y$ from a
certain point on.
Then $h(x) > h(y),$ where $h$ is as in~\eqref{Eq:h}.
\end{proposition}

\begin{proofof}
By Theorem~\ref{Thm:A9a} there exists a vector
$v$ such that $x \prec v$ and $v \leq y$.
Since $x \prec v$ and $h$ is strictly Schur-concave,
we have that $h(x) > h(v)$.
We justify this claim. 
We certainly have $h(x) \geq h(v)$.
If we had $h(x) = h(v)$,
then (by definition) $v$ would have to be
a permutation of $x$. But $x$ has distinct
coordinates so we would have $x = v$
and thus $x_i \leq y_i$ for all $i$.
Now $f$ is strictly decreasing for $x \geq 1$
thus $f(x_i) \geq f(y_i)$
and moreover at least one strict
inequality holds by the initial assumption on $k$.
Therefore
\[
h(x) = h(x_1, \ldots, x_s) = \prod_{\sigma = 1}^s f(x_{\sigma}) > 
\prod_{\sigma = 1}^s f(y_{\sigma}) = h(y),
\]
as wanted.

\noindent We may therefore assume that $h(x) > h(v)$.
A similar argument as above shows that
$v \leq y$ plus monotonicity of $f$ implies
that $h(v) \geq h(y)$.
In conclusion we have $h(x) > h(v) \geq h(y)$,
proving our assertion.
\end{proofof}

\end{document}